\documentclass[12pt,a4paper,reqno]{amsart}
\pdfoutput=1
\usepackage{amsfonts}
\usepackage{amsthm}
\usepackage{amsmath}
\usepackage{amssymb}
\usepackage{mathabx}
\usepackage{bm}
\usepackage{amscd}
\usepackage[latin2]{inputenc}
\usepackage{t1enc}
\usepackage[mathscr]{eucal}
\usepackage{indentfirst}
\usepackage{graphicx}
\usepackage{graphics}
\usepackage{pict2e}
\usepackage{epic}
\numberwithin{equation}{section}
\usepackage[margin=2.2cm]{geometry}
\usepackage{epstopdf} 
\usepackage[colorlinks,linkcolor=blue]{hyperref}

\usepackage[capitalise,noabbrev]{cleveref}

\crefformat{equation}{(#2#1#3)}
\crefmultiformat{equation}{(#2#1#3)}{ and~(#2#1#3)}{, (#2#1#3)}{ and~(#2#1#3)}
\crefrangeformat{equation}{(#3#1#4) to~(#5#2#6)}

\allowdisplaybreaks

\theoremstyle{plain}
\newtheorem{Thm}{Theorem}[section]
\newtheorem*{Thm*}{Theorem}
\newtheorem{Lem}[Thm]{Lemma}

\newtheorem{Prop}[Thm]{Proposition}

\theoremstyle{definition}

\newtheorem{Rem}[Thm]{Remark}
\newtheorem{?}[Thm]{Problem}

\newcommand{\p}{\partial}
\newcommand{\R}{\mathbb{R}}
\newcommand{\e}{\varepsilon}

\newcommand{\jump}[1]{\left\ldbrack#1\right\rdbrack}

\newcommand{\ub}{\bar{u}}
\newcommand{\vb}{\bar{v}}
\newcommand{\pb}{\bar{p}}
\newcommand{\thetab}{\bar{\theta}}
\newcommand{\Eb}{\bar{E}}

\newcommand{\vt}{\tilde{v}}
\newcommand{\ut}{\tilde{u}}
\newcommand{\pt}{\tilde{p}}
\newcommand{\Et}{\tilde{E}}
\newcommand{\thetat}{\tilde{\theta}}

\newcommand{\Rt}{\tilde{R}}
\newcommand{\Lt}{\tilde{L}}
\newcommand{\error}{\mathfrak{E}}

\newcommand{\vtt}{\breve{v}}
\newcommand{\utt}{\breve{u}}
\newcommand{\ptt}{\breve{p}}
\newcommand{\Ett}{\breve{E}}
\newcommand{\thetatt}{\breve{\theta}}
\newcommand{\Ltt}{\breve{L}}

\newcommand{\X}{\mathcal{X}}
\newcommand{\Y}{\mathcal{Y}}
\newcommand{\Z}{\mathcal{Z}}
\newcommand{\C}{\mathcal{C}}

\newcommand{\abs}[1]{\left\lvert#1\right\rvert}

\newcommand{\norm}[1]{\left\lVert#1\right\rVert}

\let\oldtocsection=\tocsection

\let\oldtocsubsection=\tocsubsection

\let\oldtocsubsubsection=\tocsubsubsection

\renewcommand{\tocsection}[2]{\hspace{0em}\oldtocsection{#1}{#2}}
\renewcommand{\tocsubsection}[2]{\hspace{1em}\oldtocsubsection{#1}{#2}}
\renewcommand{\tocsubsubsection}[2]{\hspace{2em}\oldtocsubsubsection{#1}{#2}}

\begin{document}


\title[Composite wave of two viscous shocks under periodic perturbations]{Periodic perturbations of a composite wave of two viscous shocks for 1-D full compressible Navier-Stokes equations}

\author[Q. Yuan]{Qian YUAN}
\address[Q. Yuan]{Institute of Applied Mathematics, Academy of Mathematics and System Science, Chinese Academy of Sciences, Beijing, China}
\email{qyuan@amss.ac.cn}

\author[Y. Yuan]{Yuan Yuan$ ^* $}
\address[Y. Yuan]{South China Research Center for Applied
	Mathematics and Interdisciplinary Studies, South China Normal University \\
	Guangzhou 510631, China. }
\email{yyuan2102@m.scnu.edu.cn}

\thanks{$ ^* $Corresponding author.}
\thanks{Qian Yuan is supported by the China Postdoctoral Science Foundation funded projects 2019M660831 and 2020TQ0345. The research of Yuan Yuan is supported by the National Natural Science Foundation of China (Grants No. 11901208 and 11971009), and the Natural Science Foundation of Guangdong Province, China (Grant No. 2021A1515010247).}

\maketitle

\begin{abstract} 
This paper is concerned with the asymptotic stability of a composite wave of two viscous shocks under spatially periodic perturbations for the 1-D full compressible Navier-Stokes equations. 
It is proved that as time increases, the solution approaches the background composite wave with a shift for each shock, where the shifts can be uniquely determined if both the periodic perturbations and strengths of two shocks are small. 
The key of the proof is to construct a suitable ansatz such that the anti-derivative method works.

\end{abstract}

\tableofcontents


\section{Introduction}

The one-dimensional full compressible Navier-Stokes (N-S) equations in the Lagrangian coordinates read
\begin{equation}\label{NS}
\begin{cases}
\p_t v - \p_x u = 0, & \quad \\
\p_t u + \p_x p(v,\theta) = \mu \p_x \left( \frac{\p_x u}{v} \right), & \quad \\
\p_t E + \p_x \left(p(v,\theta) u\right) = \kappa\p_x \left(\frac{\p_x\theta}{v}\right) + \mu \p_x \left(\frac{u\p_x u}{v}\right), & \quad
\end{cases} x\in \R, t>0,
\end{equation}
where $ v(x,t)>0 $ is the specific volume, $ u(x,t)\in\R $ is the velocity, $ \theta(x,t)>0 $ is the absolute temperature, the pressure $ p(v,\theta) $ satisfies $ p(v,\theta) = \frac{R\theta}{v}, $ 
and the total energy is given by $ E = e + \frac{1}{2} u^2, $ where the internal energy $ e $ is $ e = \frac{R}{\gamma-1} \theta + \text{const.} $

When $ \mu=0 $ and $ \kappa = 0, $ \cref{NS} is the full compressible Euler equations. This hyperbolic system has rich wave phenomena such as the shock, rarefaction wave, contact discontinuity and their compositions, which are called Riemann solutions, satisfying the initial data
\begin{equation*}
	(v,u,E)(x,0) = 
	\begin{cases}
	(\vb_l,\ub_l,\Eb_l), \qquad x<0, \\
	(\vb_r,\ub_r,\Eb_r), \qquad x>0,
	\end{cases}
\end{equation*}
where $ \vb_{l,r}>0, \ub_{l,r}, \Eb_{l,r}>0 $ are constants. 
This paper is concerned about a composite wave of two shocks, i.e. there is an intermediate state $ \left(\vb_m,\ub_m,\Eb_m\right) $ connecting the left state $ \left(\vb_l,\ub_l,\Eb_l\right) $ by a 1-shock and the right state $ \left(\vb_r,\ub_r,\Eb_r\right) $ by a 3-shock. More precisely, by denoting $ \jump{\cdot}_1 $ and $ \jump{\cdot}_3 $ as the jumps when crossing the 1-shock and 3-shock, respectively (e.g. $ \jump{v}_1 = \vb_m- \vb_l $ and $ \jump{v}_3 = \vb_r - \vb_m, $ etc), these constants satisfy the Rankine-Hugoniot conditions, 
\begin{equation}\label{RH}
	\begin{aligned}
		-s_i \jump{v}_i - \jump{u}_i & = 0, \\
		-s_i \jump{u}_i + \jump{p}_i & = 0, \\
		-s_i \jump{E}_i + \jump{pu}_i & = 0,
	\end{aligned} \qquad \text{for } i=1,3,
\end{equation}
and the Lax entropy conditions,
\begin{equation}\label{entropy}
	\begin{aligned}
		& \lambda_1(\vb_m,\thetab_m) <s_1<\lambda_1(\vb_l,\thetab_l), \quad  \lambda_3(\vb_r,\thetab_r) <s_3<\lambda_3(\vb_m,\thetab_m),
	\end{aligned}
\end{equation}
where $ s_1<0 $ and $ s_3>0 $ are the shock speeds of 1-shock and 3-shock, and $ \lambda_1(v,\theta) = -\sqrt{\frac{\gamma p(v,\theta)}{v}} $ and $ \lambda_3(v,\theta) = -\lambda_1(v,\theta) $ are the first and third eigenvalues of the full compressible Euler equations, respectively.

For the compressible N-S equations \cref{NS}, it is well known that if the initial data tends to constant states at the far field, i.e.
\begin{equation}\label{const-end}
	(v_0,u_0, E_0)(x) \rightarrow \begin{cases}
		\left(\vb_l,\ub_l,\Eb_l\right) & \quad \text{as } x\rightarrow -\infty,\\
		\left(\vb_r,\ub_r,\Eb_r\right) & \quad \text{as } x\rightarrow +\infty,
	\end{cases}
\end{equation}
the large time behavior of the solution is governed by the viscous version of the corresponding Riemann solution.  
The viscous version of an $ i $-shock is a traveling wave solution $ \left(v_i^S, u_i^S,E_i^S\right)(x-s_it) $ to \cref{NS}, satisfying 
\begin{equation}\label{ode}
	\begin{cases}
		-s_i \left(v_i^S\right)'(x) - \left(u_i^S\right)'(x) =0, \\
		-s_i \left(u_i^S\right)'(x) + \left( p\left(v_i^S,\theta_i^S\right) \right)'(x) = \mu \Big(\frac{\left(u_i^S\right)'}{v_i^S}\Big)'(x),  \\
		-s_i \left(E_i^S\right)'(x) + \left( p\left(v_i^S,\theta_i^S\right) u_i^S \right)'(x) = \kappa \Big(\frac{\left(\theta_i^S\right)'}{v_i^S}\Big)'(x) + \mu \Big(\frac{u_i^S\left(u_i^S\right)'}{v_i^S}\Big)'(x),
	\end{cases}
\end{equation}
with $ \theta_i^S = \frac{\gamma-1}{R} \big[E_i^S-\frac{1}{2} \left(u_i^S\right)^2 \big], $ and 
\begin{align*}
	& \left(v_1^S, u_1^S,E_1^S\right)(x) \to (\vb_l,\ub_l,\Eb_l) \quad \big(\text{resp. } (\vb_m,\ub_m,\Eb_m) \big) \quad \text{ as } x\to -\infty \ (\text{resp. } +\infty), \\
	& \left(v_3^S, u_3^S,E_3^S\right)(x) \to (\vb_m,\ub_m,\Eb_m) \quad \big(\text{resp. }  (\vb_r,\ub_r,\Eb_r) \big) \quad \text{ as } x\to -\infty \ (\text{resp. } +\infty).
\end{align*}

In this paper, we are concerned with the stability of the composite wave of two viscous shocks under periodic perturbations, i.e. we consider a Cauchy problem for \cref{NS} with the initial data 
\begin{equation}\label{ic}
	(v,u,E)(x,0) = (v_0,u_0,E_0)(x), \quad x\in\R,
\end{equation}
satisfying
\begin{equation}\label{end-behavior}
  (v_0,u_0, E_0)(x) \rightarrow 
  \begin{cases}
		\left(\vb_l,\ub_l,\Eb_l\right)+\left(\phi_{0l}, \psi_{0l}, w_{0l}\right)(x) & \quad \text{as } x\rightarrow -\infty,\\
		\left(\vb_r,\ub_r,\Eb_r\right)+(\phi_{0r}, \psi_{0r},w_{0r})(x) & \quad \text{as } x\rightarrow +\infty,
	\end{cases}
\end{equation}
where the constants $ (\vb_{l,r}, \ub_{l,r}, \Eb_{l,r}) $ satisfy \cref{RH,entropy}, and $ \left(\phi_{0l,0r}, \psi_{0l,0r}, w_{0l,0r} \right) $ are periodic functions with period $ \pi_{l,r}>0 $ and have zero averages, i.e.
\begin{equation}\label{zero-ave}
\int_{0}^{\pi_{l,r}} \left(\phi_{0l,0r}, \psi_{0l,0r}, w_{0l,0r} \right)(x) dx =0.
\end{equation}


It is the most important feature of the nonlinear hyperbolic equations that no matter how smooth and small the initial data is, the classical solution may blow up, that is, the shock waves may appear in finite time. For the 1-d hyperbolic equations, many literatures have shown that the shock waves possess strong structural stability under localized (e.g. compactly supported) perturbations; see \cite{Oleinik1960,Liu1977a}. 
For the Navier-Stokes equations, due to the effect of viscosity, the perturbed shock wave time-asymptotically tends to its viscous version, a viscous shock, which is a smooth travelling wave solution to the compressible N-S equations, connecting the shock states at the far field and travelling with the shock speed. 
The first result about the stability of viscous shocks owns to I'lin-Ole\v{\i}nik \cite{Oleinik1960} for the 1-d scalar viscous conservation laws, where the approach is based on a maximum principle for the anti-derivative variables of the perturbations. 
For the systems, \cite{Matsumura1985,Kawashima1985,Goodman1986} used the energy method to prove the stability of a single viscous shock provided that the shock-strength is small and the initial perturbation carries no excessive mass. This zero-mass condition was then successfully removed by Liu \cite{Liu1985}, Szepessy-Xin \cite{Szepessy1993} and Liu-Zeng \cite{Liu2015}, by introducing diffusion waves propagating along other families of characteristics and establishing their point-wise estimates. 
If the shock-strength is arbitrarily large, \cite{Zumbrun1998,MZ2004,HRZ2006} showed the nonlinear stability if a spectral stability holds true, which was then verified by the works \cite{HLyZ2009,HLaZ2010,BZ2016} for the Navier-Stokes equations, based on numeric analysis or high Mach numbers. Recently, with the aid of the effective velocity, \cite{He2019a} successfully used the elementary energy method to obtain the nonlinear stability of the large-amplitude viscous shock for the isentropic Navier-Stokes equations. 
For a composite wave of two weak viscous shocks of the full compressible N-S equations, Huang-Matsumura \cite{Huang2009} utilized the energy method to achieve the nonlinear stability under $ H^1(\R) \cap L^1(\R) $ perturbations. 
We also refer to \cite{MNrare1986,LX1997,HXY2004,Zeng2009,HLM2010} for the stability results of other Riemann solutions such as rarefaction waves, contact discontinuities and other composite waves, in which the initial perturbations are at least in the $ H^1(\R) $ space.

On the other hand, the study of the periodic solutions to the hyperbolic conservation laws is also important and interesting, where the solutions have infinite oscillations at the far field and therefore, there are infinitely many wave interactions.
Lax \cite{Lax1957} was the first to show algebraic decay rates of the periodic solutions to the 1-d scalar hyperbolic equations. 
Then with the aid of the novel Glimm scheme, Glimm-Lax \cite{Glimm1970} proceeded to study some 1-d $ 2\times2 $ hyperbolic systems, showing the global existence of the periodic solutions and the large time behaviors; see also Dafermos \cite{Dafermos1990}. 
However, for the compressible Euler equations, the global existence of periodic solutions is still open until now. 
In fact, the difficulty is mainly due to a resonant phenomenon proved by Majda-Rosales \cite{Majda1984} for the periodic solutions to the full compressible Euler equations, which never appears in the case for $ 2\times2 $ hyperbolic systems.
We also refer to \cite{Dafermos2013} for the asymptotic behavior of the periodic solutions to scalar convex conservation laws in multiple dimensions, and \cite{Qu2015} for a long time existence result of the periodic solutions to the 1-d full compressible Euler equations, respectively. 

The works of Lax and Glimm \cite{Lax1957,Glimm1970} reveal the asymptotic stability of constants with periodic perturbations for conservation laws.
Recently, Z. Xin and the authors \cite{Xin2018,Xin2019,Yuan2019} studied the stability of shocks and rarefaction waves with periodic perturbations for the 1-d scalar conservation laws in both inviscid and viscous cases. 
It was shown that different from the localized perturbations, there is a new phenomenon that the inviscid shock and viscous shock have different kinds of shifts under periodic perturbations, where the latter one depends on the fluxes, while the former one does not.
Huang-Yuan \cite{HY1shock} continued to study the nonlinear stability of a single viscous shock for the isentropic compressible N-S equations, in which the periodic perturbations satisfy a zero-mass type condition. 

\vspace{.2cm}
 
In this paper, we prove the nonlinear stability of a composite wave of two viscous shocks under general periodic perturbations for the full compressible N-S equations. To deal with the periodic perturbations which are not integrable on $ \R, $ the key point is to construct a suitable ansatz to carry the same oscillations as those of the solution at the far field and make the anti-derivative method work. Motivated by \cite{Xin2019,HY1shock}, the ansatz is constructed through selecting appropriate shift functions (of time) for the $ 1 $-viscous shock and $ 3 $-viscous shock, respectively, which is totally different from that in \cite{Matsumura1985,Huang2009};
see \cref{ansatz-a} for the details. It is proved that if the periodic perturbations and the strengths of shocks are both suitably small, the ansatz can be well constructed by using the implicit function theorem. With the desired ansatz, we can define the anti-derivative variables of the perturbations and use the energy method to achieve the main result, \cref{Thm}.


The rest of the paper is organized as follows. In Section 2, we introduce some notations and some useful lemmas, then present the construction of the ansatz and the main result. In Section 3, we define the anti-derivative variables of the perturbations with their error terms, then give the reformulated problem. Then the a priori estimates are shown in Section 4. 
The proofs of Lemmas \ref{Lem-shift} and \ref{Lem-F} about the shift curves and the error terms of the ansatz are supplemented in the last Section 5.

\vspace{.3cm}

\section{Ansatz and Main Results}\label{Sec-ansatz}

\subsection{Preliminaries}
In the beginning, we introduce some  notations and recall some basic properties of viscous shocks and periodic solutions to \cref{NS}.

First, since the system \cref{NS}, the R-H conditions \cref{RH}, and the entropy conditions \cref{entropy} are invariant under the Galilean transform, one can let $ u-\ub_m, \ub_l-\ub_m,0,\ub_r-\ub_m $ substitute $ u, \ub_l,\ub_m,\ub_r, $ respectively, to assume without loss of generality that
\begin{equation}\label{ub-m}
	\ub_m =0.
\end{equation}
Assume that the initial data \cref{ic} satisfies 
\begin{equation}\label{ic-L1}
	\begin{cases}
		(v_0-v^S -\phi_{0l}, u_0-u^S-\psi_{0l},E_0-E^S-w_{0l})(x) \in L^1(-\infty,0), \\
		(v_0-v^S -\phi_{0r}, u_0-u^S-\psi_{0r},E_0-E^S-w_{0r})(x) \in L^1(0,+\infty).
	\end{cases}
\end{equation}
\textbf{Notations.} 
Let $ \norm{\cdot} := \norm{\cdot}_{L^2(\R)} $ and $ \norm{\cdot}_l := \norm{\cdot}_{H^l(\R)} $ for $ l\geq 1. $ 
Denote 
\begin{equation}\label{e}
\begin{aligned}
	\e & := \sum\limits_{i=l,r} \norm{\phi_{0i}, \psi_{0i}, w_{0i} }_{H^3((0,\pi_i))}  \\
	& \quad + \int_{-\infty}^0 \big(\abs{v_0-v^S-\phi_{0l}} + \abs{u_0-u^S-\psi_{0l}} + \abs{E_0-E^S-w_{0l}} \big) dx \\
	& \quad + \int_0^{+\infty} \big(\abs{v_0-v^S-\phi_{0r}} + \abs{u_0-u^S-\psi_{0r}} + \abs{E_0-E^S-w_{0r}}\big) dx,
\end{aligned}
\end{equation} 
and the wave strengths as
\begin{equation}\label{strength}
	\begin{aligned}
		& \delta_1 := \abs{\vb_m-\vb_l}, \quad \delta_3 := \abs{\vb_r-\vb_m} \quad \text{and } \quad \delta:= \min\{ \delta_1, \delta_3 \}.
	\end{aligned}
\end{equation}
Then it follows from \cref{RH} that
\begin{equation}\label{strengh-1}
	\begin{aligned}
	& c \delta_1 \leq \abs{\ub_l}, \abs{\thetab_m-\thetab_l} \leq C \delta_1, \\
	& c \delta_3 \leq \abs{\ub_r}, \abs{\thetab_r-\thetab_m} \leq C \delta_3,
	\end{aligned}
\end{equation}
here and hereafter we let $ 0<c<1<C $ denote generic constants, independent of $ \e,\delta $ and $ t. $
As in \cite{Huang2009}, we also assume 
\begin{equation}\label{uniform}
	0< \max \{\delta_1, \delta_3\} \leq C\delta \quad \text{ as } \max \{\delta_1, \delta_3\} \to 0,
\end{equation}
which means that the strengths of two shocks are comparable. 

\begin{Lem}[\cite{Kawashima1985,Huang2009}]\label{Lem-shocks}
	Under the conditions \cref{RH,entropy}, assume that  $ \gamma\in(1,2] $ and \cref{uniform} holds with $ \delta>0 $ being small. Then \cref{ode} admits a viscous shock wave $ \big(v_i^S, u_i^S,E_i^S\big)(x-s_it) $ satisfying $ (u^S_i)' <0 $ for $ i=1,3. $
	Moreover, there exist constants $ c_0>0 $ and $ C>0, $ independent of $ x $ and $ \delta, $ such that
	\begin{align*}
		\abs{\left(v_1^S(x)-\vb_m, u_1^S(x), \theta_1^S(x)-\thetab_m\right)} \leq C \delta_1 e^{-c_0\delta \abs{x}} & \quad \text{for } x>0, \\
		\abs{\left(v_3^S(x)-\vb_m, u_3^S(x), \theta_3^S(x)-\thetab_m\right)} \leq C \delta_3 e^{-c_0\delta \abs{x}} & \quad \text{for } x<0, \\
		\abs{ \frac{d}{dx}\left(v_i^S, u_i^S, \theta_i^S \right)(x) } \leq C \delta_i^2 e^{-c_0\delta \abs{x}} & \quad \text{for } x\in\R, \quad i=1,3.
	\end{align*}
\end{Lem}
When $ A $ and $ B $ represent either $ v,u,E $ or $ \theta, $ it follows from \cref{uniform} and \cref{Lem-shocks} that 
\begin{equation}\label{cross}
	\abs{ \p_x^k \big[\big(A_1^S(x-s_1t) - \bar{A}_m\big) (B_3^S(x-s_3t)-\bar{B}_m\big)\big] } \leq C \delta^2 e^{-c_1\delta t -c_0\delta \abs{x}}, \quad k=0,1,
\end{equation}
where $ c_1 = c_0 \min\{\abs{s_1},s_3\}. $

Now we give some properties of the periodic solutions to \cref{NS}.
\begin{Lem}\label{Lem-periodic}
	Assume that $ (v_0,u_0,E_0)(x)\in H^k\big((0,\pi)\big) $ with $ k\geq 2 $ is periodic with period $ \pi>0 $ and average $ (\vb,\ub,\Eb). $ Then there exists $ \e_0>0 $ such that if 
	\begin{equation}
	\e_1:=\norm{(v_0,u_0,E_0)-(\vb,\ub,\Eb) }_{H^k((0,\pi))} \leq \e_0,
	\end{equation}
	there exits a unique periodic solution $$ (v,u,E)(x,t) \in C\big(0,+\infty;H^k((0,\pi)) \big) $$ to \cref{NS} with the initial data $ (v,u,E)(x,0) = (v_0,u_0,E_0)(x), $ which has the average $ (\vb,\ub,\Eb), $ and satisfies
	\begin{equation}\label{decay-per}
	\norm{(v,u,E)-(\vb,\ub,\Eb)}_{H^k((0,\pi))}(t) \leq C \e_1 e^{-2\alpha t}, \quad t\geq 0,
	\end{equation}
	where the constants $ C>0 $ and $ \alpha >0 $ are independent of $ \e_1 $ and $ t. $
\end{Lem}
The proof of \cref{Lem-periodic} is based on standard energy method with the aid of Poincar\'{e} inequality, which is left in the appendix for easing reading.

\vspace{.2cm}

\subsection{Ansatz}
In this paper, the \textit{ansatz} $(\vt, \ut, \Et)$ is constructed so that the anti-derivative method is available, even though the initial perturbation in \cref{end-behavior} is not integrable on $ \R. $
Motivated by \cite{Xin2019,HY1shock}, it is plausible that the solution $ (v,u,E) $ of \cref{NS,end-behavior} tends to the periodic solutions $ \left( v_{l,r}, u_{l,r},E_{l,r}\right) $ of \cref{NS} as $ x \to \mp \infty $ for all $ t\geq 0, $ which have the periodic initial data 
\begin{equation}\label{sol-peri}
	\left( v_i, u_i,E_i\right)(x,0) = (\vb_i,\ub_i,\Eb_i) + \left(\phi_{0i},\psi_{0i},w_{0i}\right)(x) \qquad \text{for } i=l,r;
\end{equation}
see the existence of periodic solutions in \cref{Lem-periodic}.
To use the energy method, the ansatz is expected to carry the same oscillations as those of the solution to \cref{NS} with \cref{end-behavior} at the far field. Following the idea of \cite{Xin2019,HY1shock}, we use the background viscous shocks to connect two periodic solutions $ \left( v_{l,r}, u_{l,r},E_{l,r}\right), $ and also a proper linear diffusion wave to carry excessive mass.


For $ i=l,r, $ let $ \theta_i(x,t) := \frac{\gamma-1}{R} \left( E_i - \frac{1}{2} u_i^2\right)(x,t), $
and define the perturbations of the periodic solutions as
\begin{equation}\label{phi-psi-lr}
\begin{aligned}
\left( \phi_i, \psi_i, w_i \right)(x,t) &:= \left(v_i,u_i,E_i \right)(x,t) - \left(\vb_i, \ub_i, \Eb_i\right), \\
\zeta_i(x,t) &:= \theta_i(x,t) - \thetab_i,
\end{aligned}
\end{equation}
which satisfies that $ \int_0^{\pi_i} \left( \phi_i, \psi_i, w_i \right)(x,t) dx = 0 $ for all $ t\geq 0; $ see \cref{Lem-periodic}. 


For the viscous shocks $ \left(v_1^S,u_1^S,E_1^S\right) $ and $ \left(v_3^S,u_3^S,E_3^S\right), $ define
\begin{align*}
g_1(x) & := \frac{v_1^S(x)-\vb_l}{\jump{v}_1}= \frac{u_1^S(x)-\ub_l}{-\ub_l}, \qquad h_1(x) := \frac{E_1^S(x)-\Eb_l}{\jump{E}_1},  \\
g_3(x) & := \frac{v_3^S(x)-\vb_m}{\jump{v}_3}= \frac{u_3^S(x)}{\ub_r}, \qquad \qquad  h_3(x) := \frac{E_3^S(x)-\Eb_m}{\jump{E}_3},
\end{align*}
where the two equalities follow from \cref{RH,ode}.
It is straightforward to check that $ 0<g_i(x),h_i(x)<1 $ and $ g_i'(x),h_i'(x) >0 $  for $ i=1,3. $

Now we are ready to \textbf{construct the ansatz}. Let $ \X(t), \Y(t), \Z(t) $ be three $ C^1 $ curves on $ [0,+\infty) $ and $ \sigma\in\R $  be a constant, all of which will be determined later. 

Set
\begin{equation}\label{ansatz-a}
\begin{aligned}
v^\sharp(x,t) :=~ & v_l(x,t) \left[ 1- \tau^1_{\X}(g_1)(x,t) \right] + \vb_m \left[ \tau^1_{\X}(g_1)(x,t) - \tau^3_{\X+\sigma}(g_3)(x,t) \right] \\
& + v_r(x,t) \tau^3_{\X+\sigma}(g_3)(x,t), \\
u^\sharp(x,t) :=~ & u_l(x,t) \left[ 1-\tau^1_{\Y}(g_1)(x,t)\right] + u_r(x,t) \tau^3_{\Y+\sigma}(g_3)(x,t), \\
E^\sharp(x,t) :=~ & E_l(x,t) \left[ 1- \tau^1_{\Z}(h_1)(x,t)\right] + \Eb_m \left[ \tau^1_{\Z}(h_1)(x,t) - \tau^3_{\Z+\sigma}(h_3)(x,t) \right] \\
& + E_r(x,t) \tau^3_{\Z+\sigma}(h_3)(x,t),
\end{aligned}
\end{equation}
and
\begin{equation}\label{theta-a}
\theta^\sharp(x,t):= \frac{\gamma-1}{R} \big[ E^\sharp - \frac{1}{2} \left(u^\sharp\right)^2 \big](x,t),
\end{equation} 
where the two shift operators $ \tau^1 $ and $ \tau^3 $ are defined as
\begin{align*}
	\tau^i_b(A)(x,t) := A(x-s_it-b(t)), \quad i=1,3,
\end{align*}
where $ s_i $ is the speed of i-shock, and $ A=A(x) $ and $ b=b(t) $ are any measurable functions.
As in \cite{Liu1985,Huang2009}, for general perturbations of viscous shocks, one should consider a diffusion wave propagating along the second family of characteristics $ r_2 = \big(1, 0, \frac{\pb_m}{\gamma-1}\big)^T. $
Let $ \eta \in \R $ be a constant to be determined later. 

Set the \textit{ansatz} as
\begin{equation}\label{ansatz-tilde}
	\begin{aligned}
		\vt :=& v^{\sharp} +\Theta, \quad 
		\ut := u^{\sharp}+ a \p_x \Theta, \quad
		\Et := E^{\sharp} +\frac{\pb_m}{\gamma-1} \Theta,
	\end{aligned}
\end{equation}
where $ \Theta(x,t) =\frac{\eta}{\sqrt{4\pi a(1+t)}}e^{-\frac{x^2}{4a(1+t)}} $ is a smooth diffusion wave, satisfying
\begin{align}
	& \p_t \Theta =a\p_x^2\Theta \quad \text{with} \quad a=\frac{(\gamma-1)\kappa}{\gamma R \vb_m}>0, \quad \int_\R \Theta(x,t) dx \equiv \eta. \label{Thet}
\end{align} 
Note that the ansatz \cref{ansatz-tilde} tends to the periodic solutions $\left( v_{l,r}, u_{l,r}, E_{l,r}\right)$ as $x \to \mp\infty $ for all $ t\geq 0. $ 
For later use, let
\begin{align}
	\thetat := & \frac{\gamma-1}{R} \big(\Et-\frac{1}{2}\ut^2\big) = \theta^{\sharp}+\frac{\pb_m}{R}\Theta -\frac{a (\gamma-1)}{R} \Big(\frac{a\abs{\p_x\Theta}^2}{2}+u^\sharp \p_x\Theta \Big), \label{thet} \\
	\pt := & \frac{R\thetat}{\vt} = p^\sharp - \frac{\Theta}{\vt} \left(p^\sharp-\pb_m\right) - \frac{a(\gamma-1)}{\vt} \Big(\frac{a \abs{\p_x\Theta}^2}{2}+u^\sharp \p_x\Theta \Big). \label{pt}
\end{align}

\vspace{0.2cm}

The whole remaining part of this subsection is devoted to determining the parameters of the ansatz, namely, the curves $ \X(t), \Y(t), \Z(t) $ and numbers $ \sigma, \eta, $ 
so that the anti-derivative method is available (\cref{Lem-variables}).
Once they are determined, we can state the main result of this paper, \cref{Thm}, which is placed in the next subsection.  
 
By plugging the ansatz $ (v^\sharp, u^\sharp, E^\sharp) $ into \cref{NS} with direct calculations, one can arrive at
\begin{equation}\label{eq-ansatz}
\begin{cases}
\p_t v^\sharp - \p_x  u^\sharp = \p_x F_{1,1} + f_{1,2} + \X' f_{1,3}, & \\
\p_t  u^\sharp + \p_x p(v^\sharp,\theta^\sharp) - \mu \p_x \big( \frac{\p_x u^\sharp}{v^\sharp} \big) = \p_x F_{2,1} + f_{2,2}+ \Y' f_{2,3}, & \\
\p_t E^\sharp + \p_x \big( p(v^\sharp, \theta^\sharp) u^\sharp \big) - \kappa \big( \frac{\p_x \theta^\sharp}{v^\sharp} \big) - \mu \p_x \big(\frac{u^\sharp \p_x u^\sharp}{v^\sharp}\big) = \p_x F_{3,1} + f_{3,2}+\Z' f_{3,3},
\end{cases}
\end{equation}
where
\begin{equation}\label{source-1}
\begin{cases}
F_{1,1} = u_l \left[ \tau^1_{\Y} (g_1) - \tau^1_{\X} (g_1) \right] - u_r \left[ \tau^3_{\Y+\sigma} (g_3) - \tau^3_{\X+\sigma} (g_3) \right], \\
f_{1,2} = \left[ s_1(v_l - \vb_m) + u_l \right] \tau^1_{\X} (g_1') - \left[ s_3(v_r - \vb_m) + u_r \right] \tau^3_{\X+\sigma} (g_3'), \\
f_{1,3}  = (v_l - \vb_m) \tau^1_{\X} (g_1') - (v_r - \vb_m) \tau^3_{\X+\sigma} (g_3'),
\end{cases}
\end{equation}
\begin{equation}\label{source-2}
\begin{cases}
F_{2,1} = p(v^\sharp,\theta^\sharp) - p(v_l,\theta_l) \left[ 1-\tau^1_{\Y} (g_1) \right]  - p(v_r,\theta_r) \tau^3_{\Y+\sigma} (g_3) \\  
\qquad \quad -\mu\big[\frac{\p_x u^\sharp}{v^\sharp} - \frac{\p_x u_l}{v_l} ( 1-\tau^1_{\Y} (g_1) ) - \frac{\p_x u_r}{v_r} \tau^3_{\Y+\sigma} (g_3)\big], \\
f_{2,2} = \big[ s_1 u_l - p(v_l,\theta_l) + \mu \frac{\p_x u_l}{v_l}  \big] \tau^1_{\Y} (g_1')  \\
\qquad\quad - \big[ s_3 u_r - p(v_r,\theta_r) + \mu \frac{\p_x u_r}{v_r} \big] \tau^3_{\Y+\sigma}(g_3'), \\
f_{2,3} = u_l \tau^1_{\Y} (g_1') - u_r \tau^3_{\Y+\sigma} (g_3'), 
\end{cases}
\end{equation}
\begin{equation}\label{source-3}
\begin{cases}
F_{3,1} =  p(v^\sharp, \theta^\sharp) u^\sharp - p(v_l, \theta_l) u_l \left[1-\tau^1_{\Z}(h_1) \right] - p(v_r, \theta_r) u_r \tau^3_{\Z+\sigma}(h_3) \\
\qquad - \kappa \big[ \frac{\p_x \theta^\sharp}{v^\sharp} - \frac{\p_x \theta_l}{v_l} \left( 1- \tau^1_{\Z}(h_1) \right) - \frac{\p_x \theta_r}{v_r} \tau^3_{\Z+\sigma}(h_3) \big] \\
\qquad - \mu \big[ \frac{u^\sharp \p_x u^\sharp}{v^\sharp} - \frac{u_l\p_x u_l}{v_l} \left( 1- \tau^1_{\Z}(h_1) \right) - \frac{u_r\p_x u_r}{v_r} \tau^3_{\Z+\sigma}(h_3) \big], \\
f_{3,2} = \big[ s_1 (E_l-\Eb_m) - p(v_l,\theta_l)u_l + \kappa\frac{\p_x\theta_l}{v_l} + \mu \frac{u_l\p_x u_l}{v_l} \big] \tau^1_{\Z}(h_1') \\
\qquad - \big[ s_3(E_r-\Eb_m) - p(v_r,\theta_r)u_r + \kappa\frac{\p_x\theta_r}{v_r} + \mu \frac{u_r \p_x u_r}{v_r} \big] \tau^3_{\Z+\sigma}(h_3'), \\
f_{3,3} = (E_l-\Eb_m) \tau^1_{\Z}(h_1') -  (E_r-\Eb_m) \tau^3_{\Z+\sigma}(h_3').
\end{cases}
\end{equation}
It is noted that since $ (v^\sharp,u^\sharp,E^\sharp,\theta^\sharp) $ tends to $ (v_{l,r},u_{l,r},E_{l,r},\theta_l) $ as $ x\rightarrow \mp\infty, $  one can verify easily that each $ F_{i,1}(x,t) ~(i=1,2,3) $ vanishes as $ \abs{x} \to \infty $ for all $ t\geq 0. $ To make the system \cref{eq-ansatz} as a conservative form, the curves $ \X, \Y $ and $ \Z $ should satisfy
\begin{equation}\label{ode-shift}
\begin{aligned}
\X'(t) = - \frac{ \int_\R f_{1,2}(x,t) dx }{ \int_\R f_{1,3}(x,t) dx }, ~
\Y'(t) = - \frac{ \int_\R f_{2,2}(x,t) dx }{ \int_\R f_{2,3}(x,t) dx }, ~ \Z'(t) = - \frac{ \int_\R f_{3,2}(x,t) dx }{ \int_\R f_{3,3}(x,t) dx },
\end{aligned}
\end{equation}
where the denominators in \cref{ode-shift} are away from zero if the initial periodic perturbations $ \left(\phi_{0i}, \psi_{0i}, w_{0i} \right)~(i=l,r) $ are small (see \cref{Lem-periodic}). The curves $ \X,\Y $ and $ \Z $ can be uniquely determined as long as the corresponding initial data $ \X_0,\Y_0 $ and $ \Z_0 $ are given. More precisely, it holds that 
\begin{Lem}\label{Lem-shift}
	Assume that  \cref{RH,zero-ave} hold. Then there exists an $ \e_0>0 $ such that if
	\begin{equation*}
    \sum_{i=l,r} \norm{\phi_{0i}, \psi_{0i}, w_{0i} }_{H^3((0,\pi_i))}<\e < \e_0,
	\end{equation*} 
	then given any constant triple $ \left(\X_0, \Y_0, \Z_0\right), $ there exists a unique solution $ (\X, \Y, \Z)(t) \in C^1[0,+\infty) $ to the system \cref{ode-shift} with the initial data $ (\X, \Y,\Z)(0)=(\X_0, \Y_0,\Z_0), $ satisfying that
	\begin{equation*}
	\abs{\left(\X',\Y',\Z'\right)(t)} + \abs{\left(\X,\Y,\Z\right)(t)-\left(\X_\infty, \Y_\infty, \Z_\infty \right)} \leq C\e e^{-2\alpha t}, \qquad t\geq 0,
	\end{equation*}	
	where the constant $ \alpha>0 $ is independent of $ \e $ and $ t. $ Moreover, the corresponding constant locations $ \X_\infty, \Y_\infty, \Z_\infty $ can be computed (in terms of 
	the constants $ \sigma, \X_0,\Y_0,\Z_0 $ and the periodic solutions \cref{sol-peri}) as follows,
	\begin{align}
		\X_\infty =~& \X_0 + 
		\frac{1}{\vb_r-\vb_l} \Bigg\{
		\int_{-\infty}^{0} \left[ \phi_{0l}(x) g_1(x-\X_0) - \phi_{0r}(x) g_3(x-\X_0-\sigma) \right] dx  \label{X-inf} \\
		& - \int_{0}^{+\infty} \left[ \phi_{0l}(x) \left(1-g_1(x-\X_0)\right) - \phi_{0r}(x) \left(1-g_3(x-\X_0-\sigma)\right) \right] dx \notag \\
		& + \frac{1}{\pi_l} \int_{0}^{\pi_l} \int_{0}^{x} \phi_{0l}(y)dy dx - \frac{1}{\pi_r} \int_{0}^{\pi_r} \int_{0}^{x} \phi_{0r}(y)dy dx \Bigg\} \notag \\
		:=~ & H_1(\X_0,\sigma), \notag
	\end{align}
	\begin{align}
	\Y_\infty =~ & \Y_0 +	
	\frac{1}{\ub_r-\ub_l} \Bigg\{
	\int_{-\infty}^{0} \left[ \psi_{0l}(x) g_1(x-\Y_0) - \psi_{0r}(x) g_3(x-\Y_0-\sigma) \right] dx \label{Y-inf} \\
	& - \int_{0}^{+\infty} \left[ \psi_{0l}(x) \left(1-g_1(x-\Y_0)\right) - \psi_{0r}(x) \left(1-g_3(x-\Y_0-\sigma)\right) \right] dx \notag \\
	& + \frac{1}{\pi_l} \int_{0}^{\pi_l} \int_{0}^{x} \psi_{0l}(y)dydx - \int_{0}^{+\infty} \frac{1}{\pi_l} \int_{0}^{\pi_l} \left[p(v_l, \theta_l)-p(\vb_l,\thetab_l)\right] dx dt \notag \\ 
	& -\frac{1}{\pi_r} \int_{0}^{\pi_r} \int_{0}^{x} \psi_{0r}(y)dydx  + \int_{0}^{+\infty} \frac{1}{\pi_r} \int_{0}^{\pi_r} \left[p(v_r,\theta_r)-p(\vb_r,\thetab_r)\right] dx dt \notag \\
	& + \frac{\mu}{\pi_l} \int_{0}^{\pi_l} \big[\log\left(\vb_l+\phi_{0l} \right)-\log(\vb_l)\big] dx - \frac{\mu}{\pi_r} \int_{0}^{\pi_r} \big[\log\left(\vb_r+\phi_{0r} \right)-\log(\vb_r)\big] dx \Bigg\} \notag \\
	:=~ & H_2(\Y_0, \sigma),  \notag
	\end{align}
	and
	\begin{align}
		\Z_\infty =~& \Z_0 + \frac{1}{\Eb_r-\Eb_l} \Bigg\{
		\int_{-\infty}^{0} \left[ w_{0l}(x) h_1(x-\Z_0) - w_{0r}(x) h_3(x-\Z_0-\sigma) \right] dx \label{Z-inf} \\
		& - \int_{0}^{+\infty} \left[ w_{0l}(x) \left(1-h_1(x-\Z_0)\right) - w_{0r}(x) \left(1-h_3(x-\Z_0-\sigma)\right) \right] dx \notag \\
		& + \frac{1}{\pi_l} \int_{0}^{\pi_l} \int_{0}^{x} w_{0l}(y)dy dx - \frac{1}{\pi_r} \int_{0}^{\pi_r} \int_{0}^{x} w_{0r}(y)dy dx \notag \\
		& + \int_0^{+\infty} \frac{1}{\pi_l} \int_0^{\pi_l} \Big[ \kappa \frac{\p_x\theta_l}{v_l} + \mu \frac{u_l\p_xu_l}{v_l} - \big( p(v_l,\theta_l) u_l - p(\vb_l,\thetab_l) \ub_l \big) \Big] dx dt \notag \\
		& - \int_0^{+\infty} \frac{1}{\pi_r} \int_0^{\pi_r} \Big[ \kappa \frac{\p_x\theta_r}{v_r} + \mu \frac{u_r\p_x u_r}{v_r} - \big( p(v_r,\theta_r) u_r - p(\vb_r,\thetab_r) \ub_r \big) \Big] dx dt
		\Bigg\}, \notag \\
		:=~& H_3(\Z_0,\sigma).  \notag
	\end{align}
\end{Lem}
Note that due to \cref{Lem-periodic}, all the integrals in \cref{X-inf,Y-inf,Z-inf} are bounded, thus $ \X_\infty, \Y_\infty $ and $ \Z_\infty $ are well-defined.
Since the proof of \cref{Lem-shift} is similar to that in \cite{Xin2019,HY1shock}, we place it in the last \cref{Sec-shift-F} for easy reading.

\vspace{0.2cm}

Define the composite wave of 1- viscous shock and 3- viscous shock with the corresponding shifts $ b=b(t) $ and $ d=d(t) $ as 
\begin{equation}\label{shocks}
	\begin{aligned}
		v^S_{(b,d)} & := \tau^1_b\left(v_1^S\right) -\vb_m + \tau^3_d\left(v_3^S\right), \\
		u^S_{(b,d)} & := \tau^1_b\left(u_1^S\right) + \tau^3_d\left(u_3^S\right), \\
		E^S_{(b,d)} & := \tau^1_b\left(E_1^S\right) -\Eb_m + \tau^3_d\left(E_3^S\right), \\
		\theta_{(b,d)}^S &:= \frac{\gamma-1}{R} \Big[E^S_{(b,d)} - \frac{1}{2} \left(u^S_{(b,d)}\right)^2\Big] \\
		&=\tau^1_b\left(\theta_1^S\right)-\thetab_m +\tau^3_d\left(\theta_3^S\right) - \frac{\gamma-1}{R} \tau^1_b\left(u_1^S\right)\tau^3_d\left(u_3^S\right).
	\end{aligned}
\end{equation}
For convenience, we omit the lower index $ (b,d) $ above when $ b=d\equiv 0. $
Moreover, we denote $ A\sim B $ when
\begin{equation}\label{rel-1}
	\norm{A-B}_{L^\infty(\R)} \leq C \e e^{-\alpha t } + C\delta^{\frac{3}{2}} e^{-c_1\delta t}, \quad t>0
\end{equation}
holds, and denote $ A\approx B $ as in \cite{Huang2009}, when the pointwise estimate
\begin{equation}\label{rel-2}
	\abs{A-B} \leq C \big(\delta^2+ \abs{\eta} \delta^{\frac{3}{2}}\big) e^{-c \delta t-c \delta \abs{x}} + C \frac{ \abs{\eta}}{(1+t)^{\frac{3}{2}}} e^{-\frac{c\abs{x}^2}{1+t}} + C(\delta+ \abs{\eta} ) e^{-c t-c \abs{x}}
\end{equation}
holds.
Thus, by Lemmas \ref{Lem-periodic} and \ref{Lem-shift} and \cref{cross}, the functions given in \cref{ansatz-a} satisfies
\begin{align*}
v^\sharp(x,t) & \sim v_1^S(x-s_1t-\X_\infty)+v_3^S(x-s_3t-\X_\infty-\sigma)-\vb_m = v^S_{(\X_\infty, \X_\infty+\sigma)}(x,t), \\
u^\sharp(x,t) & \sim u_1^S(x-s_1t-\Y_\infty)+ u_3^S(x-s_3t-\Y_\infty-\sigma) =  u^S_{(\Y_\infty, \Y_\infty+\sigma)}(x,t), \\
E^\sharp(x,t) & \sim E_1^S(x-s_1t-\Z_\infty)+ E_3^S(x-s_3t-\Z_\infty-\sigma)-\Eb_m= E^S_{(\Z_\infty, \Z_\infty+\sigma)}(x,t),
\end{align*}

\textbf{Constraints from coinciding limits.}
From \cite{Huang2009}, it is plausible to require $ \X_\infty = \Y_\infty = \Z_\infty, $ denoted by $ \xi. $ Otherwise, neither
\begin{align*}
\left( v_1^S(x-s_1t-\X_\infty), u_1^S(x-s_1t-\Y_\infty), E^S_1(x-s_1t-\Z_\infty) \right)
\end{align*}
nor
\begin{align*}
\left( v_3^S(x-s_3t-\X_\infty-\sigma), u_3^S(x-s_3t-\Y_\infty-\sigma), E^S_3(x-s_3t-\Z_\infty-\sigma) \right)
\end{align*}
is a traveling wave solution to \cref{NS}.
Thus, by \cref{Lem-shift}, one has three constraints on the five free variables $ \xi,\sigma, \X_0,\Y_0 $ and $ \Z_0 $ as
\begin{equation}\label{constrain-1}
\xi = H_1(\X_0, \sigma) = H_2(\Y_0,\sigma) = H_3(\Z_0,\sigma).
\end{equation}
Under the condition \cref{constrain-1}, it follows from Lemmas \ref{Lem-shocks} and \ref{Lem-shift} that
\begin{equation}\label{equiv}
\begin{aligned}
v^\sharp & \sim v^S_{(\X,\X+\sigma)} \sim v^S_{(\Y,\Y+\sigma)} \sim v^S_{(\Z,\Z+\sigma)} \sim v^S_{(\xi,\xi+\sigma)}, \\
u^\sharp & \sim u^S_{(\X,\X+\sigma)} \sim u^S_{(\Y,\Y+\sigma)} \sim u^S_{(\Z,\Z+\sigma)} \sim u^S_{(\xi,\xi+\sigma)}, \\
E^\sharp & \sim E^S_{(\X,\X+\sigma)} \sim E^S_{(\Y,\Y+\sigma)} \sim E^S_{(\Z,\Z+\sigma)} \sim E^S_{(\xi,\xi+\sigma)}, \\
\theta^\sharp & \sim \theta^S_{(\X,\X+\sigma)} \sim \theta^S_{(\Y,\Y+\sigma)} \sim \theta^S_{(\Z,\Z+\sigma)} \sim \theta^S_{(\xi,\xi+\sigma)}.
\end{aligned} \qquad \text{for all } t\geq 0,
\end{equation}


\textbf{Constraints from zero masses.}
From the equations \cref{NS}, \cref{Lem-F,Thet}, one can get that the perturbations $ v-\vt, u-\ut $ and $ E-\Et $ carry zero masses for all $ t > 0, $ as long as their initial data satisfy
\begin{equation}\label{0mass-ic}
\int_\R \left(v-\vt\right)(x,0) dx = 0, \quad \int_\R \left(u-\ut\right)(x,0) dx = 0, \quad \int_\R \big(E-\Et\big)(x,0) dx =0.
\end{equation} 
Then by direct calculations, the first identity in \cref{0mass-ic} gives that
\begin{align}
\eta & = \int_\R \left(v_0(x)-v^\sharp(x,0)\right) dx \notag \\
& = \int_\R \big[ v_0(x) - v_1^S(x-\X_0) - v_3^S(x-\X_0-\sigma) + \vb_m \notag \\
& \quad - \phi_{0l}(x) \left(1-g_1(x-\X_0)\right) - \phi_{0r}(x) g_3(x-\X_0-\sigma) \big] dx \notag \\
& = \int_\R \left[v_1^S(x)-v_1^S(x-\X_0) \right] dx + \int_\R \left[v_3^S(x)-v_3^S(x-\X_0-\sigma)\right] dx  \notag \\
& \quad + \int_{-\infty}^0 \left( v_0-v^S-\phi_{0l} \right)(x) dx + \int_0^{+\infty} \left( v_0-v^S-\phi_{0r} \right)(x) dx \notag \\
& \quad + \int_{-\infty}^0 \left[ \phi_{0l}(x) g_1(x-\X_0) - \phi_{0r}(x) g_3(x-\X_0-\sigma) \right] dx \notag \\
& \quad - \int_0^{+\infty} \left[ \phi_{0l}(x) (1-g_1(x-\X_0)) - \phi_{0r}(x) \left(1-g_3(x-\X_0-\sigma\right) \right] dx. \notag
\end{align}
This, together with \cref{X-inf}, yields that
\begin{align}
\eta =~& (\vb_m-\vb_l) \X_0 + (\vb_r-\vb_m)(\X_0+\sigma) + \int_{-\infty}^0 \left(v_0-v^S-\phi_{0l} \right)(x) dx \notag \\
& + \int_0^{+\infty} \left(v_0-v^S-\phi_{0r} \right)(x) dx + (\vb_r-\vb_l)(\xi-\X_0) \notag \\
& - \frac{1}{\pi_l} \int_0^{\pi_l}\int_0^x \phi_{0l}(y)dydx + \frac{1}{\pi_r}\int_0^x \int_0^{\pi_r} \phi_{0r}(y)dydx \notag \\
=~ & \left(\vb_r-\vb_l\right) \xi + \left(\vb_r-\vb_m\right) \sigma + \C_1,  \label{0mass-1}
\end{align}
where the constant $ \C_1 $ denotes the sum of the four integrals, which is independent of the six variables, $ \X_0,\Y_0,\Z_0,\sigma,\xi $ and $ \eta. $ By similar calculations, it follows from \cref{Y-inf,Z-inf} that
\begin{align}
0 =~& \int_\R \left(u_0(x)-u^\sharp(x,0)\right) dx = \left(\ub_r-\ub_l\right) \xi + \left(\ub_r-\ub_m\right)\sigma + \C_2, \label{0mass-2} \\
\frac{\pb_m}{\gamma-1} \eta =~& \int_\R \left(E_0(x)-E^\sharp(x,0)\right) dx = \left(\Eb_r-\Eb_l\right) \xi + \left(\Eb_r-\Eb_m\right)\sigma + \C_3, \label{0mass-3}
\end{align}
where
\begin{align*}
\C_2 = & \int_{-\infty}^0 \left(u_0-u^S-\psi_{0l}\right)(x) dx + \int_0^{+\infty} \left(u_0-u^S-\psi_{0r}\right) (x) dx  \\
& - \frac{1}{\pi_l} \int_{0}^{\pi_l} \int_{0}^{x} \psi_{0l}(y)dydx + \int_{0}^{+\infty} \frac{1}{\pi_l} \int_{0}^{\pi_l} \left[p(v_l, \theta_l)-p(\vb_l,\thetab_l)\right] dx dt \\ 
& + \frac{1}{\pi_r} \int_{0}^{\pi_r} \int_{0}^{x} \psi_{0r}(y)dydx - \int_{0}^{+\infty} \frac{1}{\pi_r} \int_{0}^{\pi_r} \left[p(v_r,\theta_r)-p(\vb_r,\thetab_r)\right] dx dt \\
& - \frac{\mu}{\pi_l} \int_{0}^{\pi_l} \big[\log\left(\vb_l+\phi_{0l} \right)-\log(\vb_l)\big] dx + \frac{\mu}{\pi_r} \int_{0}^{\pi_r} \big[\log\left(\vb_r+\phi_{0r} \right)-\log(\vb_r)\big] dx, \\
\C_3 =& \int_{-\infty}^0 \left(E_0-E^S-w_{0l}\right)(x) dx + \int_0^{+\infty} \left(E_0-E^S-w_{0r}\right) (x) dx \\
& - \frac{1}{\pi_l} \int_{0}^{\pi_l} \int_{0}^{x} w_{0l}(y)dy dx + \frac{1}{\pi_r} \int_{0}^{\pi_r} \int_{0}^{x} w_{0r}(y)dy dx \\
& - \int_0^{+\infty} \frac{1}{\pi_l} \int_0^{\pi_l} \Big[ \kappa \frac{\p_x\theta_l}{v_l} + \mu \frac{u_l\p_xu_l}{v_l} - \Big( p(v_l,\theta_l) u_l - p(\vb_l,\thetab_l) \ub_l \Big) \Big](x,t) dx dt \\
& + \int_0^{+\infty} \frac{1}{\pi_r} \int_0^{\pi_r} \Big[ \kappa \frac{\p_x\theta_r}{v_r} + \mu \frac{u_r\p_x u_r}{v_r} - \Big( p(v_r,\theta_r) u_r - p(\vb_r,\thetab_r) \ub_r \Big) \Big](x,t) dx dt,
\end{align*}
both of which are independent of $ \X_0,\Y_0,\Z_0,\sigma,\xi $ and $ \eta. $

Collecting \cref{constrain-1,0mass-1,0mass-2,0mass-3}, the six free variables $ \X_0,\Y_0,\Z_0,\sigma,\xi $ and $ \eta $ should satisfy the following six equalities,
\begin{equation}\label{constrain}
\begin{cases}
\xi- H_1(\X_0,\sigma) = 0, \\
\xi- H_2(\Y_0,\sigma) = 0, \\
\xi- H_3(\Z_0,\sigma) = 0, \\
(\vb_r - \vb_l) \xi + (\vb_r-\vb_m)\sigma -\eta + \C_1 = 0, \\
\left(\ub_r-\ub_l\right) \xi + \left(\ub_r-\ub_m\right)\sigma + \C_2 = 0, \\
\left(\Eb_r-\Eb_l\right) \xi + \left(\Eb_r-\Eb_m\right)\sigma - \frac{\pb_m}{\gamma-1} \eta + \C_3 =0.
\end{cases}
\end{equation}
By \cref{e,Lem-periodic}, the constants $ \C_1, \C_2 $ and $ \C_3 $ satisfy 
\begin{equation}\label{small-C}
	\abs{\C_1} + \abs{\C_2} + \abs{\C_3} \leq C\e.
\end{equation} 
One can compute the Jacobian determinant of the system \cref{constrain} as
\begin{align}
& \text{det} \left[\begin{matrix}
\p_{\X_0} H_1 & 0 & 0 & \p_\sigma H_1 & -1 & 0 \\
0 & \p_{\Y_0} H_2 & 0 & \p_\sigma H_2 & -1 & 0 \\
0 & 0 & \p_{\Z_0} H_3 & \p_\sigma H_3 & -1 & 0 \\
0&0&0& \vb_r-\vb_m & \vb_r-\vb_l & -1 \\
0 &0&0 & \ub_r-\ub_m & \ub_r-\ub_l &0 \\
0 &0&0 & \Eb_r-\Eb_m & \Eb_r-\Eb_l & - \frac{\pb_m}{\gamma-1} 
\end{matrix}
\right]  \notag \\
& \quad = \p_{\X_0} H_1 \cdot \p_{\Y_0} H_2 \cdot \p_{\Z_0} H_3 \cdot \text{det} \left[\begin{matrix}
\jump{v}_3 & \jump{v}_1+\jump{v}_3 & -1 \\
\jump{u}_3 & \jump{u}_1+\jump{u}_3 &0 \\
\jump{E}_3 & \jump{E}_1+\jump{E}_3 & - \frac{\pb_m}{\gamma-1} 
\end{matrix}
\right]. \label{matrix}
\end{align}
Recall that $ r_2 = \big(1,0,\frac{\pb_m}{\gamma-1}\big)^T $ is a 2-eigenvector.
And for weak $ i $-shock, the vector $ \left(\jump{v}_i,\jump{u}_i,\jump{E}_i\right)^T $ is close to be parallel to an $ i $-eigenvector for $ i=1,3. $ Thus, when the wave strengths $ \delta_1 $ and $\delta_3 $ are both small, the determinant of the $ 3\times3 $ matrix in \cref{matrix} is nonzero. On the other hand, all the derivatives $ \p_{\X_0} H_1, \p_{\Y_0} H_2 $ and $ \p_{\Z_0} H_3 $ are non-zero, if 
\begin{equation}\label{small}
	\sum\limits_{i=l,r} \norm{\phi_{0i}, \psi_{0i}, w_{0i} }_{H^3((0,\pi_i))} < \min\big\{\abs{\vb_r-\vb_l}, \abs{\ub_r-\ub_l}, \abs{\Eb_r-\Eb_l} \big\}.
\end{equation}
In fact, it follows from \cref{X-inf} that
\begin{equation}\label{p_H}
\begin{aligned}
	\p_{\X_0} H_1 & = 1 - \frac{1}{\vb_r-\vb_l} \int_\R \left[ \phi_{0l}(x) g_1'(x-\X_0) - \phi_{0r}(x) g_3'(x-\X_0-\sigma) \right] dx, 
\end{aligned}	
\end{equation}
which gives that 
\begin{align*}
\abs{\p_{\X_0} H_1 - 1} & \leq  \frac{1}{\abs{\vb_r-\vb_l}} \big(\norm{\phi_{0l}}_{L^\infty(\R)} +\norm{\phi_{0r}}_{L^\infty(\R)} \big).
\end{align*}
The proof of $ \p_{\Y_0} H_2 $ and $ \p_{\Z_0} H_3 $ are similar.
Therefore, with \cref{Lem-periodic}, we have the following lemma.
\begin{Lem}\label{Lem-variables}
	Assume that \cref{RH,ic-L1,uniform} hold and $ \left(\phi_{0l,0r}, \psi_{0l,0r}, w_{0l,0r} \right) \in H^3((0,\pi_{l,r})) $ satisfy \cref{zero-ave,small}.  Then there exist $ \delta_0>0 $ and $ \e_0>0 $ such that if $ \delta<\delta_0 $ and $ \e<\e_0, $  the system \cref{constrain} admits a unique solution $ (\X_0,\Y_0,\Z_0,\sigma,\xi,\eta) \in \R^6. $ 
	Moreover, it holds that $ \abs{\eta} \leq C \e. $
\end{Lem}
In fact, it follows from the last three equalities of \cref{constrain} with \cref{uniform} that 
$$ \abs{(\delta\xi, \delta\sigma,\eta)}\leq C \abs{(\C_1,\C_2,\C_3)} \leq C\e, $$ 
which finishes the proof of \cref{Lem-variables}.

\vspace{.2cm}

Thanks to \cref{Lem-variables}, the desired ansatz \cref{ansatz-tilde} is well constructed. For convenience, we denote
\begin{equation}\label{background}
	\begin{aligned}
		& \left(V_1,U_1,E_1,\Theta_1\right)(x,t):= \left(v_1^S,u_1^S,E_1^S,\theta_1^S\right)(x-s_1t-\xi), \\
		& \left(V_3,U_3,E_3,\Theta_3\right)(x,t):= \left(v_3^S,u_3^S,E_3^S,\theta_3^S\right)(x-s_3t-\xi-\sigma), 
	\end{aligned}	
\end{equation}
and $ P_i:=p\left(V_i,\Theta_i\right) $ for $ i=1,3. $

\vspace{.2cm}

\subsection{Main result}
By \cref{ic-L1}, we can well define
\begin{equation}\label{ic-Ph-0}
	(\Phi_0,\Psi_0,W_0)(x) = \int_{-\infty}^x (v_0(y)-\vt(y,0), u_0(y)-\ut(y,0), E_0(y)-\Et(y,0)) dy.
\end{equation}
Assume that 
\begin{equation}\label{ic-Ph}
	(\Phi_0,\Psi_0,W_0) \in H^2(\R),
\end{equation}
then we are ready to present the  main result.

\begin{Thm}\label{Thm}
	Assume that \cref{RH,entropy} hold, the strengths of two shocks satisfy \cref{uniform}, and the periodic perturbations $ \left(\phi_{0l,0r}, \psi_{0l,0r}, w_{0l,0r} \right) \in H^3((0,\pi_{l,r})) $ satisfy \cref{zero-ave,small}.  Further assume that the initial data satisfies \cref{ic-L1,ic-Ph} and also $ 1<\gamma\leq 2. $
	Then there exist $ \e_0>0 $ and $ \delta_0>0 $ such that if 
	$$ \norm{(\Phi_0,\Psi_0,W_0)}_2 + 
	\e < \e_0 \quad \text{and } \quad \delta <\delta_0, $$
	where $ \e $ and $ \delta $ are given in \cref{strength,e}, respectively,
	then the problem \cref{NS,ic} admits a unique global solution $ (v,u,\theta) $ satisfying
	\begin{align*}
	 (v-\vt,u-\ut,\theta-\thetat) &\in C\left(0,+\infty;  H^1(\R) \right), \quad v-\vt \in L^2\left(0,+\infty; H^1(\R)\right), \\
	(u-\ut,\theta-\thetat) & \in L^2\left(0,+\infty; H^2(\R) \right).
	\end{align*}
	Moreover, it holds that
	\begin{equation}\label{asym}
	\norm{(v,u,\theta) - \left(V_1+V_3-\vb_m, U_1+U_3, \Theta_1+\Theta_3-\thetab_m\right)}_{L^\infty(\R)} \rightarrow 0 \quad \text{ as } t\rightarrow +\infty.
	\end{equation}
\end{Thm}

\begin{Rem}
	It is noted that the zero average condition \cref{zero-ave} is necessary for the stability of the background wave. Otherwise, by adding the constant averages of these periodic perturbations onto the constants $ (\vb_i, \ub_i,\Eb_i) $ for $ i=l,r, $ the new states may generate other kinds of Riemann solutions.
\end{Rem}

\begin{Rem}
	The ansatz $ (\vt,\ut,\thetat)$ is more complicated than that in \cite{Huang2009}. Besides adding periodic perturbations, we have to choose appropriate shift functions (of time) $\X,\Y,\Z$ for each variable $ v,u $ and $ E, $ respectively, while the ansatz in \cite{Huang2009} is a composite wave in which each of the shock waves is shifted by a constant.	
	Meanwhile, the definitions of $\C_i(i=1,2,3)$ shows that in contrast to the case of localized perturbations, besides the localized part of the initial perturbation, the periodic oscillations at infinities generate another shift to the background composite wave.
\end{Rem}


\vspace{.3cm}

\section{Reformulation of the problem}  
In this section, we reformulate the problem \cref{NS,ic} into the one for the anti-derivative variables of the perturbation $ (v-\vt,u-\ut,E-\Et). $
From \cref{eq-ansatz,ansatz-tilde}, the ansatz $(\vt,\ut,\Et)$ satisfies that
\begin{equation}\label{eq-tilde}
	\begin{cases}
		\p_t \vt-\p_x \ut= \p_x F_1,\\
		\p_t \ut+\p_x \pt =\mu \p_x \left( \frac{\p_x \ut}{\vt} \right) + \p_x F_2 + \p_x \Rt_1,\\
		\p_t \Et + \p_x ( \pt \ut)=\kappa \p_x \big( \frac{\p_x \thetat}{\vt}  \big) +\mu \p_x \left( \frac{ \ut \p_x \ut}{\vt}  \right)+\p_x F_3 + \p_x \Rt_2,
	\end{cases}
\end{equation}
where $ F_i, i=1,2,3 $ are the anti-derivative variables of the source terms in \cref{eq-ansatz}, i.e.
\begin{equation}\label{F}
\begin{aligned}
	F_1(x,t) & := F_{1,1}(x,t) + \int_{-\infty}^x f_{1,2}(y,t)dy + \X'(t)\int_{-\infty}^x f_{1,3}(y,t)dy, \\
	F_2(x,t) & := F_{2,1}(x,t) + \int_{-\infty}^x f_{2,2}(y,t)dy + \Y'(t)\int_{-\infty}^x f_{2,3}(y,t)dy, \\
	F_3(x,t) & := F_{3,1}(x,t) + \int_{-\infty}^x f_{3,2}(y,t)dy + \Z'(t)\int_{-\infty}^x f_{3,3}(y,t)dy, 
\end{aligned}
\end{equation}
and the remainders $ \Rt_1 $ and $ \Rt_2 $ are given by
\begin{equation}\label{Rt1-2}
	\begin{aligned}
		\Rt_1 & = a\p_t\Theta + \pt-p^\sharp - \mu \Big(\frac{\p_x \ut}{\vt}-\frac{u^\sharp}{v^\sharp}\Big), \\
		\Rt_2 & = \frac{a\pb_m}{\gamma-1} \p_x\Theta + \pt\ut-p^\sharp u^\sharp - \kappa \Big(\frac{\p_x\thetat}{\vt}-\frac{\p_x\theta^\sharp}{v^\sharp}\Big) - \mu \Big(\frac{\ut\p_x \ut}{\vt}-\frac{u^\sharp\p_x u^\sharp}{v^\sharp}\Big).
	\end{aligned}
\end{equation}

\begin{Lem}\label{Lem-F}
	Under the assumptions of \cref{Thm}, the anti-derivative variables \cref{F} exist and satisfy that
	\begin{equation}\label{est-F}
		\norm{ F_1 }_2 + \norm{ F_2,F_3 }_1  \leq C\e e^{- \alpha t} + C\delta^{3/2} e^{-c_1 \delta t},
	\end{equation}
	where $ \alpha>0 $ and $ c_1>0 $ are the constants in Lemmas \ref{Lem-shift} and \ref{Lem-shocks}, respectively.
\end{Lem}
The proof is based on Lemmas \ref{Lem-periodic} and \ref{Lem-shift}, and we place it in \cref{Sec-shift-F} for brevity.

Introduce 
\begin{equation}\label{ansatz-breve}
\begin{aligned}
\vtt & := V_1 + V_3 -\vb_m + \Theta, \\
\utt & := U_1 + U_3 + a\p_x\Theta, \\
\Ett & := E_1 + E_3 - \Eb_m + \frac{\pb_m}{\gamma-1} \Theta, \\
\thetatt & := \frac{\gamma-1}{R}\big(\Ett - \frac{1}{2} \utt^2\big) \quad \text{and } \quad
\ptt := \frac{R\thetatt}{\vtt}.
\end{aligned}
\end{equation}
We remark that \cref{ansatz-breve} is exactly the ansatz constructed in \cite{Huang2009}, in which the initial perturbation around the background composite wave is in the $ H^1(\R) $ space, i.e. the periodic perturbations $ \left(\phi_{0i},\psi_{0i},w_{0i}\right) $ for $ i=l,r $ vanish. Comparing \cref{ansatz-breve} with the ansatz \cref{ansatz-tilde}, when $ A $ represents either $ v,u,E,\theta $ or $ p, $ direct calculations yield that
\begin{equation}\label{equiv-1}
\tilde{A} = \breve{A} + \error \quad \text{ and } \quad \p_t^j \p_x^k \tilde{A} = \p_t^j \p_x^k \breve{A} + \error, \quad j, k = 0, 1, 2, \cdots
\end{equation}
where and hereafter we use $ \error $ to represent the error terms which satisfy the relation \cref{rel-1}, i.e. $ \error\sim 0. $

\begin{Lem}\label{Lem-Rt}
Under the assumptions of \cref{Thm}, it holds that
\begin{equation}\label{ineq-Rt}
\Rt_i\approx 0, \quad  \p_x\Rt_i \approx 0, \quad i =1,2.
\end{equation}	
\end{Lem}

\begin{proof}
From \cite{Huang2009}, when $ A $ and $ B $ represent either $ v, u, E, \theta $ or $ p, $ it holds that
\begin{align*}
	\big(A_1^S(x-s_1t) - \bar{A}_m\big) \big(B_3^S(x-s_3t)-\bar{B}_m\big) & \approx 0,\\
	\big(A_i^S(x-s_it) - \bar{A}_m\big) \Theta & \approx 0,  \quad i=1,3,
\end{align*}
and 
\begin{align*}
\p_x^k \big(A^\sharp-\bar{A}_m\big) \p_x^j \Theta = \p_x^k \big( A_1 + A_3 - \bar{A}_m + \error -\bar{A}_m \big) \p_x^j \Theta \approx \error \p_x^j \Theta \approx 0, ~ k,j =0,1,2,\cdots,
\end{align*}
where $ A_i $ represents the terms given in \cref{background}.
It follows from \cref{pt} that
\begin{align*}
\Rt_1 \approx - \frac{\Theta}{\vt} \left(p^\sharp-\pb_m\right) +  \frac{a}{\vt} \p_x^2 \Theta - \frac{\Theta}{\vt v^\sharp} \p_x u^\sharp \approx 0.
\end{align*}
By \cref{thet}, it holds that
\begin{align*}
\Rt_2 & \approx \frac{a\pb_m}{\gamma-1} \p_x\Theta + p^\sharp (\ut-u^\sharp) - \frac{\kappa \pb_m}{R\vt} \p_x\Theta + \kappa \frac{\p_x\theta^\sharp}{\vt v^\sharp} \Theta - \mu \p_x u^\sharp \Big( \frac{\ut}{\vt}-\frac{u^\sharp}{v^\sharp} \Big) \\
& \approx \frac{a\pb_m}{\gamma-1} \p_x\Theta + \pb_m a\p_x\Theta - \frac{\kappa \pb_m}{R\vb_m} \p_x\Theta = 0.
\end{align*}
It is similar to prove the derivatives $ \p_x \Rt_1 $ and $ \Rt_2. $ 
\end{proof}

Define the perturbations
\begin{align*}
\phi := v-\vt, \quad \psi:= u-\ut, \quad w := E-\Et \quad \text{and} \quad \zeta := \theta-\thetat,
\end{align*}
and the anti-derivatives
\begin{equation}\label{anti-deriv}
\begin{aligned}
\big(\Phi,\Psi,W\big)(x,t) := \int_{-\infty}^{x} (\phi,\psi,w)(y,t) dy.
\end{aligned}
\end{equation}
Let $ \Xi(x,t):=\frac{\gamma-1}{R} \left(W-\ut \Psi\right)(x,t). $
Then it holds that
\begin{equation}\label{anti-deriv-2}
\zeta = \p_x \Xi-\frac{\gamma-1}{R} \Big(\frac{1}{2} \abs{\p_x \Psi}^2 -\p_x \ut \Psi \Big).
\end{equation}
From \cref{NS,eq-tilde,ic-Ph-0}, we arrive at the reformulated problem 
\begin{equation}
\label{eq-anti-deriv}
\begin{cases}
\p_t \Phi -\p_x \Psi =-F_1,& \\
\p_t \Psi - \big( \frac{\pt}{\vt} -\frac{\mu  \p_x \ut}{\vt^2}\big) \p_x \Phi
+\frac{R}{\vt}\p_x \Xi+\frac{\gamma-1}{\vt}\p_x \ut \Psi = \frac{\mu}{\vt}\p_x^2\Psi+J_1-F_2-\Rt_1,& \\
\frac{R}{\gamma-1} \p_t \Xi + \big( \pt -\frac{\mu \p_x \ut}{\vt} \big) \p_x \Psi +\p_t \ut \Psi -\frac{\kappa(\gamma-1)}{\vt R} \p_x (  \p_x \ut \Psi) +\frac{\kappa \p_x \thetat}{ v\vt}\p_x \Phi &\\
\qquad  =\frac{\kappa}{\vt}\p_x^2 \Xi +J_2 -F_3+\ut F_2 - \Rt_2 +\ut\Rt_1, & 
\end{cases}
\end{equation}
with the initial data 
\begin{equation}\label{ic-pert}
	(\Phi,\Psi,\Xi)(x,0) = (\Phi_0, \Psi_0, \Xi_0)(x) \in H^2(\R),
\end{equation}
where $ \Xi_0 := \frac{\gamma-1}{R} \big( W_0(x)- \ut(x,0) \Psi_0(x) \big) $ and $ J_1 $ and $ J_2 $ are higher order terms given by
\begin{equation}\label{eq-J}
\begin{aligned}
J_1 & =
\frac{\gamma-1}{2\vt} (\p_x \Psi)^2 +\frac{\mu}{v \vt^2} \p_x \ut (\p_x \Phi)^2 -\frac{\mu}{v \vt} \p_x^2 \Psi\p_x \Phi 
-\big( p-\pt +\frac{\pt}{\vt} \p_x\Phi-\frac{R}{\vt} \zeta \big) \\
& =  \frac{\gamma-1}{2\vt} (\p_x \Psi)^2 +\frac{\mu}{v \vt^2} \p_x \ut (\p_x \Phi)^2 -\frac{\mu}{v \vt} \p_x^2 \Psi\p_x \Phi + \frac{\p_x \Phi}{\vt} (p-\pt), \\
J_2 & =(\pt-p) \p_x \Psi + \mu \Big( \frac{\p_xu}{v} - \frac{\p_x \ut}{\vt} \Big) \p_x \Psi -\frac{\kappa (\gamma-1)}{R\vt}\p_x\Psi\p_x^2 \Psi -\kappa\frac{\p_x \Phi \p_x (\theta-\thetat)}{v\vt}.
\end{aligned}
\end{equation}

\begin{Thm}\label{Thm-anti}
	Under the assumptions of \cref{Thm}, there exist $ \delta_0>0 $ and $ \e_0>0 $ such that if $ \delta <\delta_0 $ and
	$$ \norm{(\Phi_0,\Psi_0,W_0)}_2 + \e < \e_0, $$ 
	the problem \cref{eq-anti-deriv,ic-pert} admits a unique global solution $ (\Phi,\Psi,\Xi), $ satisfying
	\begin{align*}
		& (\Phi,\Psi,\Xi) \in C\left(0,+\infty; H^2(\R) \right), \quad \p_x \Phi \in L^2\left(0,+\infty; H^1(\R)\right), \\
		& (\p_x\Psi, \p_x\Xi) \in L^2\left(0,+\infty; H^2(\R) \right).
	\end{align*}
\end{Thm}

\begin{proof}[Proof of \cref{Thm}]
If \cref{Thm-anti} holds true, with the fact that
\begin{equation}
	\norm{(\vt, \ut,\thetat) - \left(V_1+V_3-\vb_m, U_1+U_3, \Theta_1+\Theta_3-\thetab_m\right)}_{L^\infty(\R)} \leq C\e e^{-\alpha t},
\end{equation}
it is standard (see \cite{Huang2009}) to verify \cref{asym}.
Thus, it remains to prove \cref{Thm-anti} to complete the proof of \cref{Thm}.
\end{proof}

\vspace{.3cm}

\section{A priori estimates}

Based on the standard local-existence theory, one can finish the proof of \cref{Thm-anti} if the following a priori estimate, \cref{Prop-a-priori}, holds true.

For $ T>0, $ denote
\begin{equation}
\label{a-priori-assumption}
\nu: =\sup_{t\in [0,T]} \norm{(\Phi, \Psi, \Xi)(t)}_2.
\end{equation}

\begin{Prop}[A priori estimates]\label{Prop-a-priori}
	Under the assumptions of \cref{Thm-anti}, there exist positive constants $\delta_0, \e_0 $ and $ \nu_0,$ independent of $ T, $ such that if $\delta < \delta_0, \e < \e_0 $ and $ \nu<\nu_0 $, then
	\begin{align}
		& \sup_{t\in[0,T]}\norm{ \Phi, \Psi, \Xi }^2_2 +\int_{0}^T \big(\norm{ \p_x \Phi}^2_1 + \norm{ \p_x \Psi, \p_x \Xi}^2_2 \big)dt \notag \\
		& \qquad  +\int_{0}^T \int_{\R} \left(\abs{\p_x U_1}+\abs{\p_x U_3} \right) \left(\Psi^2+\Xi^2\right) dxdt \leq C \big( \norm{\Phi_0,\Psi_0,W_0}_2^2 +
		\e +\delta^{\frac{1}{2}}\big).   \label{ineq-a-priori}
	\end{align} 	
\end{Prop}

\vspace{0.2cm}

By \cref{a-priori-assumption} and the Sobolev inequality, if $ \delta, \e $ and $ \nu $ are small, one has that 
$$ \sup\limits_{t\in[0,T]} \norm{(\rho,u,\theta)}_{L^\infty(\R)} \leq C, $$ and
\begin{align*}
	\inf_{x\in\R,t\geq0} \vt \geq \frac{\vb_m}{2}, \quad \inf_{ x\in\R,  t\in[0,T]} v \geq \frac{\vb_m}{4}, \quad \inf_{x\in\R,t\geq 0} \thetat \geq \frac{\thetab_m}{2}, \quad \inf_{x\in\R,t\in[0,T]} \theta \geq \frac{\thetab_m}{4}.
\end{align*}
Moreover, the higher-order terms $ J_1 $ and $ J_2 $ in \cref{eq-J} satisfy 
\begin{equation}
\label{ineq-J}
\begin{aligned}
& \abs{J_1} \leq C \left( \abs{\p_x \Psi}^2+ \abs{\p_x \Phi}^2 + \abs{\p_x^2 \Psi}\abs{\p_x \Phi} +\abs{\p_x \Xi}^2 + \abs{\p_x \ut}\abs{\Psi}^2\right),\\
& \abs{J_2} \leq C \left(\abs{\p_x \Phi}^2+\abs{\p_x \Psi}^2+\abs{\p_x \Xi}^2 + \abs{\p_x \ut}\abs{\Psi}^2+ \abs{\p_x \Psi}\abs{\p_x^2 \Psi} + \abs{\p_x \zeta}\abs{\p_x \Phi}\right),
\end{aligned}
\end{equation}
and it holds that
\begin{equation}\label{low-bdd}
	\begin{aligned}
		& \inf\limits_{x,t} \Big(\pt-\frac{\mu \p_x \ut}{\vt}\Big) \geq  c \inf\limits_{x,t} \thetat  -C\delta -C\varepsilon \geq \frac{c\thetab_m}{4}>0, \\
		& \inf\limits_{x,t} \Big(p^\sharp-\frac{\mu \p_x u^\sharp}{v^\sharp}\Big) \geq \frac{c\thetab_m}{4} \quad \text{ and } \quad \inf\limits_{x,t} \Big(\ptt-\frac{\mu \p_x \utt}{\vtt}\Big) \geq \frac{c\thetab_m}{4}.
	\end{aligned}	
\end{equation} 
For later use, we denote
\begin{equation*}
	\Lt := \Big(\pt-\frac{\mu \p_x \ut}{\vt} \Big)^{-1}, \quad L^\sharp := \Big(p^\sharp-\frac{\mu\p_x u^\sharp}{v^\sharp} \Big)^{-1}, \quad   \Ltt :=  \Big(\ptt-\frac{\mu\p_x\utt}{\vtt} \Big)^{-1}.
\end{equation*}
The proof of \cref{Prop-a-priori} consists of the following series of lemmas.

\begin{Lem}\label{Lem-energy-1}
	Under the assumptions of \cref{Prop-a-priori},
	there exist $\delta_0 >0,\e_0 >0 $ and $ \nu_0>0 $ such that if $ \delta <\delta_0, \e<\e_0 $ and $ \nu < \nu_0, $ then
     \begin{align}
     & \sup_{t\in[0,T]}\norm{ \Phi, \Psi, \Xi }^2 +\int_{0}^T \norm{ \p_x \Psi, \p_x \Xi}^2 dt +\int_{0}^T \int_{\R} \left(\abs{\p_x U_1}+\abs{\p_x U_3} \right) \left(\Psi^2+\Xi^2\right) dxdt \notag \\
     &\qquad \leq C \Big\{ \norm{\Phi_0, \Psi_0, W_0}^2  +  \e + \delta^{\frac{1}{2}} + \big( \nu +\delta^{\frac{1}{2}}\big) \int_{0}^T \norm{\p_x \Phi, \p_x \zeta, \p_x^2 \Psi }^2 dt \Big\}. \label{ineq-energy-1}
     \end{align} 
\end{Lem}

\begin{proof}
With direct calculations, $\cref{eq-anti-deriv}_1 \cdot \Phi + \cref{eq-anti-deriv}_2 \cdot \vt \Lt\Psi + \cref{eq-anti-deriv}_3 \cdot R\Lt^2 \Xi$ gives that 
\begin{equation}\label{eq-energy-1}
\begin{aligned}
\p_t N_1+ \sum_{i=2}^{4} N_i &= \p_x(\cdots) - F_1 \Phi + \big(J_1-F_2-\Rt_1 \big)\vt \Lt \Psi  \\
&\quad + \big( J_2-F_3+\ut F_2 -\Rt_2 + \ut \Rt_1 \big)R\Lt^2 \Xi,
\end{aligned}
\end{equation} 
where $\p_x (\cdots)$ vanishes after integration on $ \R $ and
\begin{equation}\label{N1-4}
\begin{aligned}
N_1& =\frac{1}{2} \big( \Phi^2 +\vt \Lt \Psi^2 +\frac{R^2 }{\gamma-1} \Lt^2 \Xi^2 \big),\\
N_2& = \tilde{b} \Psi^2 + \mu \p_x \Lt  \Psi\p_x \Psi +\mu \Lt (\p_x \Psi)^2 \quad \text{ with } \quad  \tilde{b} =-\frac{1}{2} \p_t(\vt \Lt) +(\gamma-1 )\Lt \p_x \ut,\\
N_3& =-\frac{R^2}{\gamma-1} \Lt \p_t\Lt \Xi^2 +\p_x \Big(\frac{\kappa R}{\vt}\Lt^2 \Big) \Xi \p_x \Xi +\frac{\kappa R}{\vt}\Lt^2 (\p_x \Xi)^2,\\
N_4&= R\big(\Lt^2 \p_t \ut -\p_x \Lt\big) \Psi \Xi +\frac{\kappa R}{v\vt } \p_x \thetat \Lt^2 \p_x \Phi \Xi  + \kappa(\gamma-1) \p_x \ut \Psi \p_x\Big( \frac{\Lt^2 }{\vt} \Xi\Big).
\end{aligned}
\end{equation}
Now we estimate the terms in \cref{eq-energy-1} one by one.
First, similar to the proof of \cref{Lem-Rt}, one can verify that
\begin{align}
& \p_t \pt \approx \p_t p^\sharp = \p_t \ptt + \error \approx \p_t \left(P_1+P_3-\pb_m \right) + \error, \notag \\
& \p_t \Big( \frac{\mu\p_x \ut}{\vt}\Big) \approx \p_t \Big(\frac{\mu\p_x u^\sharp}{v^\sharp}\Big) \approx \p_t \Big(\frac{\mu\p_x\utt}{\vtt} \Big) + \error \approx \p_t \Big(\frac{\mu\p_x U_1}{V_1}+\frac{\mu\p_x U_3}{V_3}\Big) +\error, \\
& \Lt \approx  L^\sharp \approx L_1 +L_3 +\error,  \quad \p_t \Lt \approx \p_t L_1+\p_t L_3 +\error, \quad \p_x \Lt 
\approx\p_x L_1+\p_x L_3 +\error, \label{equiv-2}
\end{align}
where $ L_i :=(P_i-\frac{\mu \p_x U_i}{V_i})^{-1}=(b_i-s_i^2 V_i)^{-1} $ with $ b_i = \pb_m + s_i^2 \vb_m $ for $ i=1,3, $
and the coefficient $ \tilde{b} $ in $ N_2 $ satisfies that
\begin{align*}
\tilde{b} & \approx -\frac{1}{2} \p_t(v^\sharp L^\sharp) +(\gamma-1 )L^\sharp \p_x u^\sharp = -\frac{1}{2} \p_t(\vtt \Ltt) +(\gamma-1 )\Ltt \p_x \utt + \error \\
& \approx \sum_{i=1,3} \big[ -\frac{1}{2} \p_t (L_iV_i) +(\gamma-1) L_i \p_x U_i \big] +\error = \frac{1}{2} \sum_{i=1,3}  \abs{\p_x U_i} L_i^2 \left( b_i-2(\gamma-1)L_i^{-1}\right)+\error. 
\end{align*}
For easy reading, we first ignore the terms arising from the relation $ ``\approx 0" $ and postpone their estimates at the end of the proof. Thus, it holds that
\begin{align*}
	\Lt & \geq \sum_{i=1,3} (\pb_m -C\delta_i)^{-1} - \norm{\error}_{L^\infty(\R)} \geq c - C \e e^{-\alpha t}-C\delta^{\frac{3}{2}} e^{-c_1\delta t}, \\
	\tilde{b} & \geq c \sum_{i=1,3} \abs{\p_x U_i} \big((3-\gamma)\pb_m-C\delta_i \big) - C \e e^{-\alpha t}-C\delta^{\frac{3}{2}} e^{-c_1\delta t},
\end{align*}
and $ \mu \p_x \Lt \Psi \p_x \Psi \approx \sum\limits_{i=1,3} \mu s_i^2 L_i^2 \p_x V_i \Psi \p_x \Psi +\error \Psi \p_x \Psi. $ Since $ \abs{\p_x V_i} \leq C\abs{\p_xU_i} \leq C\delta_i, $ then if $ \e $ and $ \delta $ are small, one can get that
\begin{equation}\label{ineq-N2}
\begin{aligned}
\int_0^T\int_\R N_2 dxdt 
& \geq c \int_0^T \int_\R (\abs{\p_x U_1}+\abs{\p_x U_3}) \abs{\Psi}^2 dx dt + c \int_0^T \norm{\p_x\Psi}^2 dt \\
& \quad - C (\e +\delta^{\frac{1}{2}}) \sup_{t\in[0,T]} \norm{\Psi}^2.
\end{aligned}
\end{equation}

By \cref{equiv-2} and the fact that $ \abs{\p_x \Theta}^2 \approx 0,~ \p_t L_i = s_i^2 L_i^2 s_i \p_x V_i = - s_i^2 L_i^2 \abs{\p_x U_i} $ for $ i=1,3, $ it holds that
\begin{align*}
-\frac{R}{\gamma-1} \Lt \p_t \Lt \Xi^2 &\geq c (\abs{\p_x U_1}+\abs{\p_x U_3}) \Xi^2 - \norm{\error}_{L^\infty(\R)} \Xi^2,\notag\\
\abs{\p_x\Big(\frac{\kappa R}{\vt}\Lt^2 \Big) \Xi \p_x \Xi} & \leq C \big(\abs{\p_x \vt} + \abs{\p_x \Lt} \big) \abs{\Xi} \abs{\p_x \Xi} \\
&\leq \frac{\kappa R}{2\vt}\Lt^2 (\p_x \Xi)^2 + C\delta \left(\abs{\p_x U_1} +\abs{\p_x U_3}  \right) \Xi^2 + \norm{\error}_{L^\infty(\R)}^2 \Xi^2. \notag
\end{align*}
Then if $ \e $ and $ \delta $ are small, it holds that
\begin{equation}\label{ineq-N3}
\int_0^T\int_\R N_3 dxdt \geq c \int_0^T \norm{\left(\abs{\p_x U_1}+\abs{\p_x U_3} \right)^{\frac{1}{2}} \Xi, \p_x\Xi }^2 dt - C \big(\e +\delta^{\frac{1}{2}}\big) \sup_{t\in[0,T]} \norm{\Xi}^2.
\end{equation}

By the facts that $ \abs{\p_x\Theta_i} \leq C\abs{\p_x U_i}, i=1,3 $ and $ \delta_0^{-\frac{1}{2}} \abs{\p_x\Theta}^2 \approx 0, $ one can verify  
\begin{align}
	& \Lt^2 \p_t \ut -\p_x \Lt = \Lt^2 \Big[ \p_t \ut +\p_x \Big(\pt -\frac{\mu \p_x \ut}{\vt}\Big) \Big] = \Lt^2 \big(\p_x F_2 +\p_x \Rt_1\big) \approx \Lt^2 \p_x F_2 =\error, \notag \\
	%
	%
	& \abs{\frac{\kappa}{v\vt } R\Lt^2 \p_x \thetat\p_x \Phi \Xi} \leq  \Big(\abs{\p_x \thetatt} + \norm{\error}_{L^\infty(\R)}\Big) \abs{\p_x\Phi} \abs{\Xi}  \notag \\
	& \qquad \leq C\delta_0^{\frac{1}{2}}  \left(\abs{\p_x U_1}+\abs{\p_x U_3}\right) \Xi^2 + C\delta_0^{\frac{1}{2}} \abs{\p_x \Phi}^2  + \norm{\error}_{L^\infty(\R)} \left(\Xi^2 + \abs{\p_x \Phi}^2\right). \notag 
\end{align} 
Thus, it holds that
\begin{equation}\label{ineq-N4}
	\begin{aligned}
		\int_0^T\int_\R \abs{N_4} dxdt  &\leq C\big(\e +\delta^{\frac{1}{2}}\big) \int_0^T \norm{\left(\abs{\p_x U_1}+\abs{\p_x U_3}\right)^{\frac{1}{2}} \left(\Psi, \Xi \right), \p_x \Phi, \p_x\Xi }^2 dt \\
		& \quad + C \big(\e+\delta^{\frac{1}{2}}\big) \sup_{t\in[0,T]} \norm{\Psi,\Xi}^2.
	\end{aligned} 
\end{equation} 

It follows from \cref{ineq-J} that
\begin{align}
\abs{\vt \Lt J_1 \Psi} &\leq C \norm{\Psi}_{L^\infty(\R)} \big[ \abs{\p_x \Psi}^2+\abs{\p_x \Phi}^2 + \abs{\p_x \Xi}^2 + \abs{\p_x^2 \Psi}^2 \notag \\
& \qquad\qquad\qquad\quad + \big(\abs{\p_x U_1} + \abs{\p_x U_3}+\norm{\error}_{L^\infty(\R)} \big) \Psi^2 \big], \label{ineq-J1}\\
\abs{R \Lt^2 J_2 \Xi}&\leq C \norm{\Xi}_{L^\infty(\R)} \big[ \abs{\p_x\Phi}^2 + \abs{\p_x\Psi}^2 + \abs{\p_x \Xi}^2 + \abs{\p_x^2 \Psi}^2 \notag \\
& \qquad\qquad\quad\qquad + \big(\abs{\p_x U_1} + \abs{\p_x U_3}+\norm{\error}_{L^\infty(\R)} \big) \Psi^2 + \norm{\p_x \zeta}^2 \big]
\label{ineq-J2}.
\end{align}
Then one has that
\begin{align}
& \int_0^T \int_\R \abs{\text{RHS of \cref{eq-energy-1}}} dxdt \notag \\
& \quad \leq C \int_0^T \norm{F_1,F_2,F_3} \norm{\Phi, \Psi, \Xi} dt + C \nu \int_0^T \Big( \norm{\p_x\left(\Phi,\Psi,\Xi\right)}^2 + \notag \\
& \qquad \qquad \norm{\p_x^2 \Psi, \p_x \zeta}^2  + \norm{\left(\abs{\p_x U_1}+\abs{\p_x U_3} \right)^{\frac{1}{2}} \Psi}^2 \Big) dt + C \big(\e+\delta^{\frac{1}{2}}\big) \sup_{t\in[0,T]} \norm{\Psi}^2 \notag \\
& \quad \leq C \big( \e+\delta^{\frac{1}{2}}\big) + C \big( \e+\delta^{\frac{1}{2}}\big) \sup_{t\in[0,T]} \norm{\Phi, \Psi, \Xi}^2 + C \nu \int_0^T \Big( \norm{\p_x\left(\Phi,\Psi,\Xi\right)}^2+ \notag \\
& \qquad \qquad \norm{\left(\abs{\p_x U_1}+\abs{\p_x U_3} \right)^{\frac{1}{2}} \Psi}^2 \Big)dt + C \nu \int_0^T \norm{\p_x^2 \Psi, \p_x \zeta}^2 dt. \label{ineq-RHS-1}
\end{align}

At last, we deal with all the error terms arising from the relation  ``$\approx$'' that were postponed in the previous estimates, denoted by $ \Rt $. Same as \cite[Lemma~3.1]{Huang2009}, the integrals of them on $ \R\times [0,T] $ can be bounded by 
\begin{align}
& \int_{0}^{T}\int_{\R} \abs{\Rt} \left( \abs{\Phi} + \abs{\Psi}+\abs{\Xi}+ \abs{\p_x \Phi}+ \abs{\p_x \Psi}+ \abs{\p_x \Xi} \right)dx  dt \notag \\
& \quad \leq C\int_{0}^{T} \Big[  (\delta^{\frac{3}{2}}+|\eta|\delta ) e^{-c\delta t} +\frac{|\eta|}{(1+t)^{ \frac{5}{4}}} + (\delta+\abs{\eta}) e^{-ct} \Big] \norm{ \Phi, \Psi, \Xi, \p_x \Phi, \p_x \Psi, \p_x \Xi }  dt \notag \\
& \quad \leq C \nu \int_0^T \Big[  (\delta^{\frac{3}{2}}+|\eta|\delta ) e^{-c\delta t} +\frac{|\eta|}{(1+t)^{ \frac{5}{4}}} + (\delta+\abs{\eta}) e^{-ct} \Big] dt \notag \\
& \quad \leq C (\delta^{\frac{1}{2}} + \abs{\eta}) \leq C (\delta^{\frac{1}{2}}+\e) . \label{ineq-approx}
\end{align}
Thus, collecting the estimates \cref{ineq-N2,ineq-N3,ineq-N4}, \cref{ineq-RHS-1,ineq-approx}, one can get \cref{ineq-energy-1} if $ \e,\delta $ and $ \nu $ are small.
 
\end{proof}

\begin{Lem}\label{Lem-energy-2}
	Under the assumptions of \cref{Prop-a-priori},
	there exist small $\delta_0 >0,\e_0 >0 $ and $ \nu_0>0 $ such that if $ \delta <\delta_0, \e<\e_0 $ and $ \nu < \nu_0, $ then
	\begin{align}
		& \sup_{t\in[0,T]} (\norm{\Phi}^2_1 + \norm{ \Psi, \Xi }^2) +\int_{0}^T \int_{\R} \left(\abs{\p_x U_1}+\abs{\p_x U_3} \right) \left(\Psi^2+\Xi^2\right) dxdt \label{ineq-energy-2} \\
		&\quad +\int_{0}^T \norm{ \phi, \psi, \p_x \Xi}^2 dt \leq C \Big\{ \norm{\Phi_0}^2_1 + \norm{ \Psi_0, W_0}^2 + \e +\delta^{\frac{1}{2}}  + \big( \nu +\delta^{\frac{1}{2}}\big) \int_{0}^T \norm{ \p_x \zeta, \p_x \psi }^2 dt\Big\}. \notag
	\end{align}
\end{Lem}
\begin{proof}
Taking $ \p_x \Psi=\p_t \Phi +F_1 $ into $\eqref{eq-anti-deriv}_2, $ 
and then multiplying the resulting equation by $\p_x \Phi, $ one has that
\begin{equation}\label{eq-energy-2}
	\begin{aligned}
	&\p_t\Big(\frac{\mu}{2\vt} \abs{\p_x \Phi}^2 \Big) - \p_t\Big(\frac{\mu}{2\vt} \Big) \abs{\p_x \Phi}^2 - \p_x \Phi\p_t \Psi+\frac{1}{\vt \Lt} \abs{\p_x \Phi}^2 \\
	&\qquad = \Big(\frac{R}{\vt} \p_x \Xi +\frac{\gamma-1}{\vt}\p_x \ut \Psi -J_1+\Rt_1  -\frac{\mu  }{\vt}\p_xF_1+F_2\Big)\p_x \Phi.
\end{aligned}
\end{equation}
Since $ \p_x \Phi \p_t \Psi =\p_t(\p_x\Phi \Psi)-\p_x (\p_t\Phi\Psi) + \abs{\p_x \Psi}^2 -F_1 \p_x \Psi $ and by the fact that if $\delta >0 $ and $ \e >0 $ are small,
\begin{align*}
	 - \p_t\Big(\frac{\mu}{2\vt} \Big)+\frac{1}{2\vt \Lt} &= \frac{\pt}{2\vt}+\frac{\mu (\p_t\vt-\p_x\ut)}{2\vt^2} \geq c\thetab_m - C\norm{ \p_x F_1}_{L^{\infty}(\R)}\geq c\thetab_m - C(\e+\delta^{\frac{3}{2}}),
\end{align*}
integrating \cref{eq-energy-2} over $ \R\times[0,T] $ yields that
\begin{align}
	& \int_\R \big(\frac{\mu}{2\vt} \abs{\p_x \Phi}^2 -\p_x \Phi \Psi \big)(x,T) dx + \int_0^T \int_\R \frac{1}{4\vt\Lt}\abs{\p_x \Phi}^2 dx dt \notag\\
	& \qquad \leq C \big(\norm{\Phi_0}_1^2+\norm{\Psi_0}^2 \big) + C \int_0^T \norm{\p_x \Psi, \p_x\Xi}^2 dt  \notag\\
	& \qquad \quad + C \delta \int_0^T \int_\R \big( \abs{\p_x U_1}+ \abs{\p_x U_3} \big) \Psi^2 dx dt  + C \int_0^T \norm{\error}_{L^\infty(\R)} \norm{\Psi}^2 dt \notag \\
	& \qquad \quad  +C \int_0^T \Big(\norm{\Rt_1}^2 + \norm{F_1}_1^2 + \norm{F_2}^2\Big) dt  + C \nu \int_0^T  \norm{J_1}_{L^{1}(\R)} dt. \label{ineq-lem2} 
\end{align}
Similar to the proof of \cref{Lem-energy-1}, we can use Lemmas \ref{Lem-F} and \ref{Lem-Rt} and \cref{ineq-J} to estimate the last three integrals on the RHS of \cref{ineq-lem2} to get that 
\begin{align*}
& \sup_{t\in[0,T]} \norm{\p_x\Phi}^2 + \int_0^T \norm{\p_x\Phi}^2 dt  \leq C \Big\{ \sup_{t\in[0,T]} \norm{\Psi}^2 + \norm{\Phi_0}_1^2 + \norm{\Psi_0}^2 +  \int_0^T \norm{\p_x (\Psi, \Xi)}^2 dt \\
& \quad + \int_0^T \int_\R \big( \abs{\p_x U_1}+ \abs{\p_x U_3} \big) \Psi^2 dx dt  +\e^2 +\delta^2 + \int_0^T \norm{\p_x^2\Psi}^2 dt
\Big\}.
\end{align*} 
This, together with \cref{Lem-energy-1}, yields \cref{ineq-energy-2}.

\end{proof}

\begin{Lem}\label{Lem-energy-3}
	Under the assumptions of \cref{Prop-a-priori},
	there exist small $\delta_0 >0,\e_0 >0 $ and $ \nu_0 >0 $ such that if $ \delta<\delta_0, \e<\e_0 $ and $ \nu<\nu_0, $ then
	\begin{equation}\label{ineq-energy-3}
		\sup_{t\in[0,T]}  \norm{\phi, \psi, \zeta }^2 +\int_{0}^T \norm{ \p_x \psi, \p_x \zeta}^2 dt \leq C \big( \norm{\Phi_0, \Psi_0, W_0}^2_1 + \e +\delta^{\frac{1}{2}}\big).
	\end{equation}
\end{Lem}

\begin{proof}
	Subtracting \eqref{eq-tilde}$ _2 $ from \eqref{NS}$ _2 $ gives that
	\begin{equation}\label{eq-energy-3}
	\p_t \psi -\mu \p_x \Big(\frac{\p_x\psi}{v}\Big) = -\p_x \Big(p-\pt  + \mu \frac{\p_x\ut}{v\vt} \phi + F_2+\Rt_1 \Big).	
	\end{equation}
	Multiplying $ \psi $ on \cref{eq-energy-3} and integrating the resulting equation on $ \R\times[0,T] $ with the fact that $ \norm{p-\pt} \leq C\norm{\zeta,\phi}, $ one has that
	\begin{equation}\label{ineq-lem3-1}
		\sup_{t\in[0,T]}\norm{\psi}^2 + \int_0^T \norm{\p_x \psi}^2 dt \leq C \norm{\Psi_0}_1^2 + C \int_0^T \norm{\zeta,\phi,F_2,\Rt_1}^2 dt.
	\end{equation}
	It follows from \cref{anti-deriv-2} that
	\begin{align*}
		\int_0^T \norm{\zeta}^2 dt
		& \leq C \int_0^T \norm{\p_x\Xi, \p_x \Psi}^2 dt + C \int_0^T \norm{\big(\abs{\p_x U_1} + \abs{\p_x U_3}\big)^{\frac{1}{2}} \Psi}^2 dt \\
		& \quad + C\big(\e + \delta^{ \frac{1}{2}}\big) \sup_{t\in[0,T]} \norm{\Psi}^2.
	\end{align*}
	Combining \cref{ineq-lem3-1} and \cref{Lem-energy-2}, if $ \e, \nu $ and $ \delta $ are small enough, one has that
	\begin{align}
		& \sup_{t\in[0,T]}\norm{\psi}^2 +\int_{0}^T \norm{ \p_x \psi}^2 dt \notag \\
		&\qquad \leq C \norm{\Phi_0,\Psi_0}_1^2 + \norm{W_0 }^2 + C\big( \nu +\delta^{\frac{1}{2}}\big) \int_{0}^T \norm{\p_x \zeta}^2 dt +C \big(\e +\delta^{\frac{1}{2}}\big). \label{ineq-energy-psi}
	\end{align}
	Similarly, subtracting \eqref{eq-tilde}$ _3 $ from \eqref{NS}$ _3 $ gives that
	\begin{align}
	& \frac{R}{\gamma-1} \p_t \zeta - \kappa \p_x \Big( \frac{\p_x\zeta}{v} \Big) = - (p-\pt) \p_x \ut - p \p_x\psi - \p_x \Big( \kappa \frac{\p_x\thetat}{v\vt} \phi + F_3 + \Rt_2 \Big) \notag \\
	& \qquad - \frac{\mu (\p_x\ut)^2}{v\vt} \phi + \frac{\mu}{v} (2\p_x\ut + \p_x\psi) \p_x\psi + \p_x \big[ \ut (F_2+\Rt_1) \big] -\p_x\ut (F_2+\Rt_1). 
	\label{eq-energy-4}
	\end{align}
	Then multiplying $ \zeta $ on \cref{eq-energy-4} and integrating the resulting equation on $ \R\times[0,T], $ one can verify that
	\begin{align*}
	\sup_{t\in[0,T]} \norm{\zeta}^2 + \int_0^T \norm{\p_x\zeta}^2 dt \leq C\norm{\zeta(x,0)}^2 + C \int_0^T \norm{\phi,\zeta,\p_x\psi,F_3, \Rt_2 ,F_2,\Rt_1}^2  dt.
	\end{align*}
	Then combining \cref{ineq-energy-psi}, Lemmas \ref{Lem-F}, \ref{Lem-Rt} and \ref{Lem-energy-2} and the fact that $ \norm{\zeta(x,0)}^2 \leq C \norm{W_0,\Psi_0}_1^2, $ one can obtain \cref{ineq-energy-3}. 
	
\end{proof}

\begin{Lem}
	Under the assumptions of \cref{Prop-a-priori},
	there exist small $\delta_0 >0,\e_0 >0 $ and $ \nu_0 >0 $ such that if $ \delta<\delta_0, \e<\e_0 $ and $ \nu<\nu_0, $ then
	$$\sup_{t\in[0,T]}  \norm{ \p_x (\phi, \psi, \zeta) }^2 +\int_{0}^T \big(\norm{\p_x\phi}^2 + \norm{ \p_x^2 (\psi, \zeta)}^2 \big) dt \leq C(\norm{\Phi_0, \Psi_0, W_0}^2_2 + \e +\delta^{\frac{1}{2}}). $$ 
\end{Lem}
\begin{proof}
	The proof is similar to that of Lemmas \ref{Lem-energy-2} and \ref{Lem-energy-3}, so it is omitted here.
\end{proof}


Once \cref{Prop-a-priori} is proved, it remains to supplement the proof of Lemmas \ref{Lem-shift} and \ref{Lem-F}.

\vspace{.3cm}

\section{Proof of Lemmas \ref{Lem-shift} and \ref{Lem-F}}\label{Sec-shift-F}

\begin{proof}[Proof of \cref{Lem-shift}]
	When $ \e = \sum_{i=l,r} \norm{\phi_{0i}, \psi_{0i}, w_{0i} }_{H^3((0,\pi_i))} $ is small, with \cref{Lem-periodic}, the existence, uniqueness and regularities of $\X,\Y,\Z$ can be easily derived from the ODEs, \eqref{ode-shift}. Besides, the exponential decay rates of $ \X',\Y' $ and $ \Z' $ can follow from \cref{RH,decay-per} directly.
	
	Now we calculate the limits of these curves. For brevity, we give only the proof of $ \X_\infty, $ since the other two , \eqref{Y-inf} and \eqref{Z-inf}, are similar to prove.
	
	For any fixed 
	$y\in[0,1], t>0$ and integer $N>0$, define the domain
	\begin{align*}
	&\Omega_{y}^N(t) := \{ (x,\tau); 0<\tau<t, \Gamma_l^N (\tau) <x< \Gamma_r^N (\tau)  \}, \\
	& \quad \text{ where } \ \Gamma_l^N (\tau) := s_1\tau+\X (\tau) +(-N+y)\pi_l, \ \Gamma_r^N (\tau):= s_3 \tau +\X (\tau)+\sigma +(N+y)\pi_r.
	\end{align*}
 Then integration by parts yields that 
	$$\lim_{N\rightarrow +\infty} \int_{0}^{1} \iint_{ \Omega_{y}^N(t) } (\p_t v^\sharp -\p_x u^\sharp) dx d\tau dy=0, $$
	which yields that
	\begin{align}
	\lim_{N\rightarrow +\infty} \int_{0}^{1}  \Big\{  &
	\int_{\Gamma_l^N(0) }^{\Gamma_r^N(0)} v^\sharp(x,0) dx 
	+ \int_0^t [(s_3+\X') v^\sharp + u^\sharp ] ( \Gamma_r^N (\tau), \tau ) d\tau \notag \\
	& - \int_{\Gamma_l^N(t) }^{\Gamma_r^N(t)} v^\sharp(x,t) dx 
	- \int_0^t [(s_1+\X') v^\sharp + u^\sharp ] ( \Gamma_l^N (\tau), \tau ) d\tau
	\Big\} dy=0. \label{integral}
	\end{align}
	For the integrals on $\tau=0$ and $\tau=t$, it holds that
	\begin{align*}
	&\int_{\Gamma_l^N(0) }^{\Gamma_r^N(0)} v^\sharp(x,0) dx- \int_{\Gamma_l^N(t) }^{\Gamma_r^N(t)} v^\sharp(x,t) dx = \sum_{i=1}^4 I_i,
	\end{align*}
	where
	\begin{align*}
	I_1 & = \int_{\Gamma_l^N(0) }^{\Gamma_r^N(0)} [ \phi_{0l} (1- \tau^1_{\X_0} (g_1))  + \phi_{0r} \tau^3_{\X_0+\sigma}(g_3) ](x,0) dx , \\
	I_2 & = \int_{\Gamma_l^N(0) }^{\Gamma_r^N(0)} [\tau^1_{\X_0}(v^S_1) + \tau^3_{\X_0+\sigma}(v^S_3)-\vb_m](x,0) dx, \\
	I_3 & = - \int_{\Gamma_l^N(t) }^{\Gamma_r^N(t)} [\phi_l (1- \tau^1_{\X} (g_1))  + \phi_r \tau^3_{\X+\sigma}(g_3)](x,t) dx, \\
	I_4 & = - \int_{\Gamma_l^N(t) }^{\Gamma_r^N(t)} [  \tau^1_{\X}(v^S_1) + \tau^3_{\X+\sigma}(v^S_3)-\vb_m ](x,t) dx.
	\end{align*}
	One can verify that 
	\begin{align*}
	I_1 & = \int_0^{\Gamma_r^N(0)} [ \phi_{0l} (1- \tau^1_{\X_0} (g_1))  - \phi_{0r} (1-\tau^3_{\X_0+\sigma}(g_3)) ](x,0) ] dx + \int_0^{\Gamma_r^N(0)} \phi_{0r}(x) dx \\
	& \quad + \int_{\Gamma_l^N(0)}^0 [-\phi_{0l}\tau^1_{\X_0} (g_1) + \phi_{0r}\tau^3_{\X_0+\sigma}(g_3) ] dx + \int_{\Gamma_l^N(0)}^0 \phi_{0l}(x) dx,
	\end{align*}
	which yields that
	\begin{align}
	\lim_{N\to +\infty} \int_0^1 I_1 dy & = \int_0^{+\infty} [ \phi_{0l} (1- \tau^1_{\X_0} (g_1))  - \phi_{0r} (1-\tau^3_{\X_0+\sigma}(g_3)) ](x,0) ] dx \notag \\
	& \quad  + \int_{-\infty}^0 [-\phi_{0l}\tau^1_{\X_0} (g_1) + \phi_{0r}\tau^3_{\X_0+\sigma}(g_3) ](x,0) dx \label{I-1} \\
	& \quad + \int_0^1 \int_0^{\X_0+y\pi_r} \phi_{0r}(x)dxdy + \int_0^1 \int_{\X_0+y\pi_l}^0 \phi_{0l}(x) dxdy. \notag
	\end{align}
	Since $\phi_{0l} $ and $ \phi_{0r}$ have zero average, then 
	\begin{align*}
	\int_{0}^{1}\int_{0}^{\X_0 + y \pi_r} \phi_{0r}(x) dx dy & = \frac{1}{\pi_r} \int_{0}^{\pi_r} \int_{0}^{y} \phi_{0r} (x) dx dy, \\
	\int_{0}^{1} \int_{\X_0+y\pi_l}^0 \phi_{0l}(x) dxdy & = -\frac{1}{\pi_l} \int_{0}^{\pi_l} \int_{0}^{y} \phi_{0l} (x) dxdy.
	\end{align*}
	Similar calculations of $ I_1 $ with \cref{Lem-periodic} yield that
	\begin{align}
		\abs{\lim_{N\to +\infty} \int_0^1 I_3 dy} \leq C e^{-\alpha t}. \label{I-3}
	\end{align}
	In addition, 
	\begin{align*}
	I_2+I_4 & = - \vb_m [\Gamma_r^N(0)-\Gamma_l^N(0)]+ \vb_m [\Gamma_r^N(t)-\Gamma_l^N(t)]\\
	&\quad - \int_{(N+y)\pi_r+\sigma}^{(s_3-s_1)t +(N+y) \pi_r+\sigma } v^S_1 (x) dx 
	+\int_{(-N+y)\pi_l-\sigma}^{-(s_3-s_1)t +(-N+y) \pi_l-\sigma } v^S_3 (x) dx.
	\end{align*}
	Since $ v^S_1(x) $ (resp. $ v_3^S(x) $)  $ \to \vb_m$ as $x \to +\infty $ (resp. $ -\infty $), then one has that
	\begin{align}
	\lim_{N\to +\infty} \int_0^1 (I_2+I_4)dy  = - \vb_m (s_3 -s_1) t. \label{I-2-4}
	\end{align}	
	Since $ (v^\sharp,u^\sharp) \to (v_r,u_r) $ as $ x \to +\infty, $ the integral on $\Gamma_r^N$ in \cref{integral} satisfies that
	\begin{align}
	& \lim_{N\rightarrow +\infty} \int_{0}^{1}  \Big\{   \int_0^t [(s_3+\X') v^\sharp + u^\sharp ] ( \Gamma_r^N (\tau), \tau ) d\tau \Big\} dy \notag \\
	& \quad =  \int_{0}^{t} \int_{0}^{1}  [(s_3+\X') v_r + u_r ] ( \Gamma_r^N(\tau), \tau ) dy d\tau \notag \\
	& \quad = (s_3 t +\X- \X_0)\vb_r  +\ub_r t. \label{int-Ga-r}
	\end{align}
	By similar calculations, the left integral on $\Gamma_l^N$ satisfies that
	\begin{equation}\label{int-Ga-l}
	\lim_{N\rightarrow +\infty} \int_{0}^{1}  \Big\{   \int_0^t [(s_1+\X') v^\sharp + u^\sharp ] ( \Gamma_l^N (\tau), \tau ) d\tau \Big\} dy = (s_1 t +\X- \X_0)\vb_l  +\ub_l t.
	\end{equation}
	Collecting \cref{I-1,I-2-4,int-Ga-r,int-Ga-l,integral} gives that
	\begin{align*}
	& \int_{0}^{+\infty} \left[ \phi_{0l}(x) \left(1-g_1(x-\X_0)\right) - \phi_{0r}(x) \left(1-g_3(x-\X_0-\sigma)\right) \right] dx \\
	& \quad -\int_{-\infty}^{0} \left[ \phi_{0l}(x) g_1(x-\X_0) - \phi_{0r}(x) g_3(x-\X_0-\sigma) \right] dx \\
	& \quad - \frac{1}{\pi_l} \int_{0}^{\pi_l} \int_{0}^{y} \phi_{0l}(x) dxdy + \frac{1}{\pi_r} \int_{0}^{\pi_r} \int_{0}^{y} \phi_{0r}(x) dxdy - \vb_m(s_3-s_1)t \\
	& \quad + \vb_r(s_3 t +\X- \X_0)  +\ub_r t - \vb_l(s_1 t +\X- \X_0) -\ub_l t=O(\e e^{-2\alpha t}),
	\end{align*}
	where $ O(e^{-2\alpha t}) $ denotes the terms satisfying $ \abs{O(e^{-\alpha t})} \leq C\e e^{-2\alpha t}. $
	This, together with the Rankine-Hugoniot conditions \cref{RH}$ _1 $ for $ i=1,3, $ yields \cref{X-inf} directly.
	
\end{proof}

\begin{proof}[Proof of \cref{Lem-F}]
Here we give only the proof of $ F_3, $ since the proofs of the other two are similar.


First, it follows from the equation of $ \Z', $ \cref{ode-shift}, that
\begin{align}
F_3(x,t) & = F_{3,1}(x,t) + \int_{-\infty}^x f_{3,2}(y,t)dy + \Z'(t)\int_{-\infty}^x f_{3,3}(y,t)dy, \label{int-min}\\
& = F_{3,1}(x,t) - \int_x^{+\infty} f_{3,2}(y,t)dy - \Z'(t)\int_x^{+\infty} f_{3,3}(y,t)dy. \label{int-plus}
\end{align}

Case 1. For $ x< s_3 t, $ we decompose $ F_3(x,t) $ according to \cref{int-min} as follows.
\begin{align}
F_3 =~& p(v^\sharp,\theta^\sharp)u^\sharp - p(v_l,\theta_l) u_l + p(\vb_l,\thetab_l)\ub_l \tau^1_{\Z}(h_1) - p(\vb_r,\thetab_r)\ub_r \tau^3_{\Z+\sigma}(h_3) \notag \\
& - \kappa\Big(\frac{\p_x \theta^\sharp}{v^\sharp} - \frac{\p_x \theta_l}{v_l}\Big) - \mu \Big( \frac{u^\sharp\p_x u^\sharp}{v^\sharp} - \frac{u_l\p_x u_l}{v_l} \Big) \label{F-3} \\
& - \left( s_1 \jump{E}_1 + p(\vb_l,\thetab_l) \ub_l \right) \tau^1_{\Z}(h_1) - \left( s_3 \jump{E}_3 - p(\vb_r,\thetab_r)\ub_r \right) \tau^3_{\Z+\sigma}(h_3)  + D, \notag
\end{align}
where the remainder $ D $ is the sum of products of some well-decaying terms, given by
\begin{align*}
D = D_1 \tau^1_{\Z}(h_1) - D_2 \tau^3_{\Z+\sigma}(h_3)  + \int_{-\infty}^{x} D_3 \tau^1_{\Z}(h_1') dy -\int_{-\infty}^{x} D_4 \tau^3_{\Z+\sigma}(h_3') dy.
\end{align*}
where 
\begin{align*}
D_1 & = p(v_l, \theta_l) u_l- p(\vb_l,\thetab_l)\ub_l -\kappa \frac{\p_x \theta_l}{v_l} -\mu \frac{u_l\p_x u_l}{v_l}, \\
D_2 & = p(v_r, \theta_r) u_r - p(\vb_r,\thetab_r)\ub_r -\kappa \frac{\p_x \theta_r}{v_r} -\mu \frac{u_r\p_x u_r}{v_r}, \\
D_3 & = s_1(E_l-\Eb_l )-p(v_l, \theta_l) u_l+ p(\vb_l,\thetab_l)\ub_l +\kappa \frac{\p_x \theta_l}{v_l} +\mu \frac{u_l\p_x u_l}{v_l} +\Z' (E_l-\Eb_m), \\
D_4 & = s_3(E_r-\Eb_r)-p(v_r, \theta_r) u_r+ p(\vb_r,\thetab_r)\ub_r +\kappa \frac{\p_x \theta_r}{v_r} +\mu \frac{u_r\p_x u_r}{v_r} +\Z' (E_r-\Eb_m),
\end{align*}
each of which is periodic, and satisfies 
$ \norm{D_i}_{W^{1,\infty}(\R)} \leq C\e e^{-2\alpha t}. $
Then it holds that
\begin{align*}
\sum_{k=0}^1 \int_{-\infty}^{s_3 t} \abs{\p_x^k D(x,t)}^2 dx & \leq C \e^2 e^{-4 \alpha t} \sum_{k=0}^1 \int_{-\infty}^{s_3 t} \left[ \abs{\p_x^k \tau^1_{\Z}(h_1)}^2+\abs{\p_x^k \tau^3_{\Z+\sigma}(h_3)}^2 \right] dy \\
& \leq C\e^2 e^{-4\alpha t} \big[ C+(s_3-s_1)t \big], \\
& \leq C\e^2 e^{-2\alpha t}.
\end{align*}
With \cref{ode}$ _3, $ one can further decompose \cref{F-3} as follows,
\begin{align}
F_3- D & =  \Big( p(v^\sharp,\theta^\sharp)u^\sharp  - p(v_l,\theta_l)u_l \Big) - \Big( p\left(v^S_{(\xi,\xi+\sigma)}, \theta^S_{(\xi,\xi+\sigma)}\right) u^S_{(\xi,\xi+\sigma)} - p(\vb_l,\thetab_l)\ub_l \Big)  \notag \\
& \quad - \kappa \Big( \frac{\p_x \theta^\sharp}{v^\sharp} - \frac{\p_x \theta_l}{v_l}  - \frac{\p_x \theta^S_{(\xi,\xi+\sigma)}}{v^S_{(\xi,\xi+\sigma)}} \Big) \notag \\
& \quad - \mu \Big( \frac{u^\sharp\p_x u^\sharp}{v^\sharp} - \frac{u_l\p_x u_l}{v_l} - \frac{u^S_{(\xi,\xi+\sigma)}\p_x u^S_{(\xi,\xi+\sigma)} }{v^S_{(\xi,\xi+\sigma)}}  \Big) \notag \\
& \quad +  p\left(v^S_{(\xi,\xi+\sigma)}, \theta^S_{(\xi,\xi+\sigma)}\right) u^S_{(\xi,\xi+\sigma)} - \tau^1_{\xi} \left(p(v_1^S, \theta_1^S) u_1^S\right) - \tau^3_{\xi+\sigma} \left(p(v_3^S, \theta_3^S) u_3^S\right)  \notag \\
& \quad - \kappa \Big[ \frac{\p_x \theta^S_{(\xi,\xi+\sigma)}}{v^S_{(\xi,\xi+\sigma)}} - \tau^1_{\xi} \Big( \frac{\p_x \theta_1^S}{v_1^S} \Big) - \tau^3_{\xi+\sigma}\Big(\frac{\p_x\theta_3^S}{v_3^S}\Big) \Big] \notag \\
& \quad - \mu \Big[\frac{u^S_{(\xi,\xi+\sigma)}\p_x u^S_{(\xi,\xi+\sigma)} }{v^S_{(\xi,\xi+\sigma)}} - \tau^1_{\xi} \Big(\frac{u_1^S\p_xu_1^S}{v_1^S} \Big) - \tau^3_{\xi+\sigma}  \Big(\frac{u_3^S\p_xu_3^S}{v_3^S} \Big) \Big] \notag \\
& \quad + s_1 \left(\tau^1_{\xi}-\tau^1_{\Z}\right)\left(E_1^S\right) + s_3 \left(\tau^3_{\xi+\sigma}-\tau^3_{\Z+\sigma}\right)\left(E_3^S\right) \notag \\
& := \sum\limits_{i=1}^7 K_i, \label{K-i}
\end{align}
where $ K_i $ represents each line on the right-hand side.
Then we estimate $ K_1 $ to $ K_7 $ one by one. 

One can verify that
\begin{align*}
K_1 & = K_{1,1} \left(v^\sharp-v_l\right) + K_{1,2} (\theta^\sharp-\theta_l) + p(v_l,\theta_l) \left(u^\sharp-u_l\right) \\
& \quad- K_{1,3} \left(v^S_{(\xi,\xi+\sigma)}-\vb_l \right) - K_{1,4} \left(\theta^S_{(\xi,\xi+\sigma)}-\thetab_l \right)  - p(\vb_l,\thetab_l) \left( u^S_{(\xi,\xi+\sigma)}-\ub_l\right).
\end{align*}
where 
\begin{align*}
K_{1,1} & = u^\sharp \int_0^1 (\p_v p)\left(v_l+\rho(v^\sharp-v_l), \theta^\sharp \right) d\rho, \\
K_{1,2} & = u^\sharp \int_0^1 (\p_\theta p)\left(v_l, \theta_l+ \rho(\theta^\sharp-\theta_l) \right) d\rho, \\
K_{1,3}& = u^S_{(\xi,\xi+\sigma)} \int_0^1 (\p_v p)\left(\vb_l+\rho(v^S_{(\xi,\xi+\sigma)}-\vb_l), \theta^S_{(\xi,\xi+\sigma)} \right) d\rho, \\
K_{1,4} & = u^S_{(\xi,\xi+\sigma)} \int_0^1 (\p_\theta p)\left(\vb_l, \thetab_l +\rho(\theta^S_{(\xi,\xi+\sigma)}-\thetab_l) \right) d\rho,
\end{align*}
By \cref{equiv}, one can verify that
\begin{align}
v^\sharp - v_l & = (\vb_m-\vb_l) \tau^1_{\xi}(g_1) + (\vb_r-\vb_m)\tau^3_{\xi+\sigma}(g_3) + J_1  = v^S_{(\xi,\xi+\sigma)} - \vb_l + J_1, \label{J-1} \\
\theta^\sharp - \theta_l 
& = \frac{\gamma-1}{R} \big( \Eb_l \tau^1_\xi(h_1)+ \Eb_r \tau^3_{\xi+\sigma}(h_3) - \Eb_m \big) \notag \\
& \quad - \frac{\gamma-1}{2R} \big(u^S_{\xi,\xi+\sigma}+ \ub_l\big) \big(- \ub_l \tau_\xi^1(g_1)+ \ub_r \tau^3_{\xi+\sigma}(g_3) \big) + J_2  \label{J-2}\\
& = \theta^S_{(\xi,\xi+\sigma)} - \thetab_l + J_2, \notag \\
u^\sharp-u_l & = - \ub_l \tau^1_{\xi}(g_1) + \ub_r \tau^3_{\xi+\sigma}(g_3) + J_3  = u^S_{(\xi,\xi+\sigma)} - \ub_l + J_3, \label{J-3}
\end{align}
where the remainders $ J_i $ for $ i=1,2,3, $ are some terms which satisfy that 
\begin{align*}
	\sum_{k=0}^2 \int_{-\infty}^{s_3t} \abs{\p_x^k J_i}^2 dx \leq C\e^2 e^{- 2\alpha t}.
\end{align*}
Thus,
\begin{align*}
K_1 & = (K_{1,1}-K_{1,3})\left(v^S_{(\xi,\xi+\sigma)}-\vb_l\right) + K_{1,1} J_1 + (K_{1,2}-K_{1,4})\left(\theta^S_{(\xi,\xi+\sigma)}-\thetab_l\right) \\
& \quad + K_{1,2} J_2 + p(v_l,\theta_l) J_3 + \left[p(v_l,\theta_l) - p(\vb_l,\thetab_l) \right] \left(u^S_{(\xi,\xi+\sigma)}-\ub_l\right).
\end{align*}
Note that $ K_{1,1}\sim K_{1,3},K_{1,2} \sim K_{1,4} $ and $ p(v_l,\theta_l) \sim p(\vb_l,\thetab_l), $ then one has that
\begin{equation}\label{K-1}
\sum_{k=0}^1 \int_{-\infty}^{s_3 t} \abs{\p_x^k K_1}^2 dx \leq C\e^2 e^{- 2\alpha t}.
\end{equation} 
Taking \cref{J-2} into the formula of $ K_2 $ gives that 
\begin{align*}
- \kappa^{-1} K_2 
= \p_x \theta_l \Big(\frac{1}{v^\sharp}- \frac{1}{v_l}\Big) + \Big( \frac{1}{v^\sharp} -\frac{1}{v^S_{(\xi,\xi+\sigma)}} \Big) \p_x \left(\theta^S_{(\xi,\xi+\sigma)}-\thetab_l\right)+\frac{\p_x J_2 }{v^\sharp}.
\end{align*}
Using \cref{Lem-periodic} and the fact that $ v^\sharp \sim v^S_{(\xi,\xi+\sigma)}, $ one can get that
\begin{equation}\label{K-2}
\sum_{k=0}^1 \int_{-\infty}^{s_3 t} \abs{\p_x^k K_2}^2 dx \leq C\e^2 e^{- 2\alpha t}.
\end{equation}
Similarly, $ K_3 $ satisfies that
\begin{align*}
-2\mu^{-1} K_3 =~& \frac{1}{v^\sharp} \p_x \big[ \left(u^\sharp\right)^2 - u_l^2 - \left(u^S_{(\xi,\xi+\sigma)}\right)^2 + \ub_l^2 \big] + \p_x u_l^2 \Big(\frac{1}{v^\sharp}-\frac{1}{v_l}\Big) \\
& + \p_x \left(u^S_{(\xi,\xi+\sigma)}\right)^2 \Big(\frac{1}{v^\sharp}-\frac{1}{v^S_{(\xi,\xi+\sigma)}}\Big).
\end{align*}
It follows from \cref{J-3} that
\begin{align*}
\left(u^\sharp\right)^2 - u_l^2 - \left(u^S_{(\xi,\xi+\sigma)}\right)^2 + \ub_l^2 = J_3 (u^\sharp+u_l) + \left(u^S_{(\xi,\xi+\sigma)}-\ub_l\right) \left(u^\sharp-u^S_{(\xi,\xi+\sigma)} + u_l -\ub_l\right).
\end{align*} 
Thus, it holds that
\begin{equation}\label{K-3}
\sum_{k=0}^1 \int_{-\infty}^{s_3 t} \abs{\p_x^k K_3}^2 dx \leq C\e^2 e^{- 2\alpha t}.
\end{equation}
For $ K_4, $ it holds that 
\begin{align*}
K_4 & = \underbrace{\left[p\left(v^S_{(\xi,\xi+\sigma)}, \theta^S_{(\xi,\xi+\sigma)}\right) - \tau^1_{\xi}\left(p(v_1^S,\theta_1^S)\right) \right]}_{K_{4,1}} \tau^1_{\xi}\left(u_1^S\right) \\
& \quad + \underbrace{\left[p\left(v^S_{(\xi,\xi+\sigma)}, \theta^S_{(\xi,\xi+\sigma)}\right) - \tau^3_{\xi+\sigma}\left(p(v_3^S,\theta_3^S)\right) \right]}_{K_{4,2}} \tau^3_{\xi+\sigma}\left(u_3^S\right),
\end{align*}
where 
\begin{align*}
K_{4,1} & = \int_0^1 (\p_v p)\left((1-\rho)\tau^1_{\xi}(v_1^S)+ \rho v^S_{(\xi,\xi+\sigma)}, \theta^S_{(\xi,\xi+\sigma)} \right) d\rho \left( v^S_{(\xi,\xi+\sigma)} - \tau^1_{\xi}(v_1^S) \right) \\
& \quad + \int_0^1 (\p_\theta p)\left(\tau^1_{\xi}(v_1^S), (1-\rho)\tau^1_{\xi}(\theta_1^S)+ \rho \theta^S_{(\xi,\xi+\sigma)} \right) d\rho \left( \theta^S_{(\xi,\xi+\sigma)} - \tau^1_{\xi}(\theta_1^S) \right),
\end{align*}
where $ v^S_{(\xi,\xi+\sigma)} - \tau^1_{\xi}(v_1^S) =  \tau^3_{\xi+\sigma}\left(v_3^S\right) -\vb_m $ and $ \theta^S_{(\xi,\xi+\sigma)} - \tau^1_{\xi}(\theta_1^S) = \tau^3_{\xi+\sigma}\left(\theta_3^S\right) -\thetab_m - \frac{\gamma-1}{R} \tau^1_{\xi}\left(u_1^S\right) \tau^3_{\xi+\sigma}\left(u_3^S\right). $
Similarly, one can verify that
\begin{align*}
K_{4,2} & = \int_0^1 (\p_v p)(\cdots) d\rho \left( \tau^1_{\xi}\left(v_1^S\right) -\vb_m \right) \\
& \quad + \int_0^1 (\p_\theta p)(\cdots) d\rho \ \big[ \tau^1_{\xi} \left( \theta_1^S\right) -\thetab_m - \frac{\gamma-1}{R} \tau^1_{\xi}\left(u_1^S\right) \tau^3_{\xi+\sigma}\left(u_3^S\right) \big]. 
\end{align*}
Then it follows from \cref{cross} that
\begin{equation}\label{K-4}
\sum_{k=0}^1 \int_\R \abs{\p_x^k K_4}^2 dx \leq C\delta^4 \int_\R e^{-2c_1\delta t - 2c \delta\abs{x} } dx \leq C\delta^3 e^{-2 c_1 \delta t}.
\end{equation}
And $ K_5 $ and $ K_6 $ satisfy that
\begin{align*}
-\kappa^{-1} K_5 =~ & - \frac{1}{v^S_{(\xi,\xi+\sigma)}} \Big[ \tau^1_{\xi} \Big(\frac{\p_x\theta_1^S}{v_1^S} \Big) \left(\tau^3_{\xi+\sigma}(v_3^S) - \vb_m\right) \\
& + \tau^3_{\xi+\sigma} \Big(\frac{\p_x\theta_3^S}{v_3^S} \Big) \left(\tau^1_{\xi}(v_1^S) - \vb_m\right) + \frac{\gamma-1}{R} \p_x \left( \tau^1_{\xi}(u_1^S) \tau^3_{\xi+\sigma}(u_3^S) \right) \Big], \\
-2\mu^{-1} K_6 =~& - \frac{1}{v^S_{(\xi,\xi+\sigma)}} \Big[ \tau^1_{\xi} \Big(\frac{\p_x \left(u_1^S\right)^2}{v_1^S} \Big) \left(\tau^3_{\xi+\sigma}(v_3^S) - \vb_m\right) \\
& + \tau^3_{\xi+\sigma} \Big(\frac{\p_x\left(u_3^S\right)^2}{v_3^S} \Big) \left(\tau^1_{\xi}(v_1^S) - \vb_m\right) -2 \p_x \left( \tau^1_{\xi}(u_1^S) \tau^3_{\xi+\sigma}(u_3^S) \right) \Big].
\end{align*}
Similar to \cref{K-4}, one can get from \cref{Lem-shocks,cross} that
\begin{equation}\label{K-5-6}
\sum_{k=0}^1 \int_\R \left(\abs{\p_x^k K_5}^2 + \abs{\p_x^k K_6}^2 \right) dx \leq C\delta^3 e^{-2c_1 \delta t}.
\end{equation}
By \cref{Lem-shift}, one can verify that
\begin{equation}\label{K-7}
\sum_{k=0}^1 \int_\R \abs{\p_x^k K_7}^2 dx \leq C\e^2 e^{-4 \alpha  t}.
\end{equation}
Collecting \cref{K-1,K-2,K-3,K-4,K-5-6,K-7} gives that
\begin{equation}\label{int-F3-min}
\sum_{k=0}^1 \int_{-\infty}^{s_3 t} \abs{\p_x^k F_3}^2 dx \leq C\e^2 e^{- 2\alpha t} + C\delta^3 e^{-2c_1\delta t}.
\end{equation}

\vspace{0.3cm}

Case 2. If $ x > s_3 t, $ one can decompose $ F_3 $ according to \cref{int-plus} and using similar arguments as in the case 1 to obtain that
\begin{equation}\label{int-F3-plus}
	\sum_{k=0}^1 \int_{s_3 t}^{+\infty} \abs{\p_x^k F_3}^2 dx \leq C\e^2 e^{-2\alpha t} + C\delta^3 e^{-2c_1\delta t}.
\end{equation}
We omit the proof in this case.

\end{proof}

\vspace{.3cm}

\appendix 

\section{Proof of \cref{Lem-periodic}}

\begin{proof}
For convenience, by using $ u-\ub $ to substitute $ u, $ one can assume that $ \ub=0. $ 

Denote the perturbation terms
\begin{align*}
& \phi(x,t) = v(x,t)-\vb, \qquad \psi(x,t) = u(x,t) - \ub =u(x,t), \\
& w(x,t) = E(x,t) - \Eb, \qquad \zeta(x,t) = \theta(x,t) - \thetab,
\end{align*}
which satisfies 
\begin{equation}\label{sys}
\begin{aligned}
\p_t \phi- \p_x \psi &=0, \\
\p_t\psi + \p_x \left(\frac{R\zeta}{v}\right) + \p_x\left(\frac{R\thetab}{v}\right) & = \mu \p_x \left(\frac{\p_x \psi}{v} \right), \\
\frac{R}{\gamma-1} \p_t \zeta + p \p_x\psi & = \kappa \p_x \left(\frac{\p_x \zeta}{v} \right) + \frac{\mu}{v} \left(\p_x\psi\right)^2.
\end{aligned}
\end{equation}
Assume that $ k\geq 2 $ and
\begin{equation}\label{apriori-per}
\nu = \sup_{t\in[0,T]} \norm{\phi, \psi, \zeta}_{H^k((0,\pi))}(t) >0
\end{equation}
is small enough.
Multiplying $ -R\thetab \left(\frac{1}{v}-\frac{1}{\vb}\right) $ on \cref{sys}$_1, $ $ \psi $ on \cref{sys}$ _2, $ and $ -\thetab \left(\frac{1}{\theta}-\frac{1}{\thetab}\right)=\frac{\zeta}{\theta} $ on \cref{sys}$ _3, $ respectively, and summing the results together yield that
\begin{align*}
& \p_t \left[ \frac{1}{2}\psi^2 + R\thetab \Phi\left(\frac{v}{\vb}\right) + \frac{R}{\gamma-1} \thetab \Phi\left(\frac{\theta}{\thetab}\right) \right] + \frac{\mu}{v} \left(\p_x \psi\right)^2  + \frac{\kappa}{v\theta} \left(\p_x \zeta\right)^2 \\
& \qquad = \p_x \left[ \frac{\mu}{v} \psi \p_x\psi + \frac{\kappa}{v\theta} \zeta\p_x\zeta - R \left( \frac{\theta}{v} - \frac{\thetab}{\vb} \right) \psi \right] + \frac{\kappa}{v\theta^2} \zeta (\p_x\zeta)^2 + \frac{\mu}{v\theta} \zeta \left( \p_x \psi \right)^2,
\end{align*}
where $ \Phi(s) = s-\ln s -1. $
It follows from \cref{apriori-per} that
\begin{equation}\label{per-est-1}
\begin{aligned}
\frac{d}{dt} \int_0^\pi \left[ \frac{1}{2}\psi^2 + R\thetab \Phi\left(\frac{v}{\vb}\right) + \frac{R}{\gamma-1} \thetab \Phi\left(\frac{\theta}{\thetab}\right) \right] dx + 2 c_1 \norm{\p_x\psi, \p_x\zeta}^2 \leq 0,
\end{aligned}
\end{equation}
for some constant $ c_1>0. $ By the conservative forms of \cref{NS}, one has that
\begin{equation}\label{per-ave}
\int_0^\pi \left( \phi, \psi, w \right)(x,t) \equiv 0, \quad t\geq 0.
\end{equation}
Then the Poincar\'{e} inequality yields that
\begin{equation}\label{poin-1}
\norm{\phi} \leq a \norm{\p_x \phi}, \norm{\psi} \leq a \norm{\p_x \psi}, \norm{w} \leq a \norm{\p_x w},
\end{equation}
for some constant $ a>0. $ This, together with $ \zeta = \frac{\gamma-1}{R} \left( w-\frac{1}{2}\psi^2 \right), $ yields that
\begin{align}
\norm{\p_x \zeta}^2 & \geq a^{-2} \norm{\zeta}^2 - \frac{(\gamma-1)^2}{4R^2a^2} \norm{\psi}_{L^\infty}^2 \norm{\psi}^2 - \frac{(\gamma-1)^2}{R^2} \norm{\psi}_{L^\infty}^2 \norm{\p_x\psi}^2 \notag \\
& \geq a^{-2} \norm{\zeta}^2 - \frac{(\gamma-1)^2}{4R^2a^2} \nu^2 \norm{\psi}^2 - \frac{(\gamma-1)^2}{R^2} \nu^2 \norm{\p_x\psi}^2. \label{poin-2}
\end{align}
Thus, by \cref{poin-1,poin-2,per-est-1}, if $ \nu>0 $ is small enough, one has that
\begin{equation}\label{per-est-2}
\begin{aligned}
\frac{d}{dt} \int_0^\pi \left[ \frac{1}{2}\psi^2 + R\thetab \Phi\left(\frac{v}{\vb}\right) + \frac{R}{\gamma-1} \thetab \Phi\left(\frac{\theta}{\thetab}\right) \right] dx + c_1 \norm{\p_x\psi, \p_x\zeta}^2 + c_2 \norm{\psi,\zeta}^2 \leq 0,
\end{aligned}
\end{equation}
By using \cref{sys}$ _1, $ \cref{sys}$ _2 $ is equivalent to 
\begin{align*}
\p_t \psi + \frac{R}{v} \p_x\zeta = \p_t \left( \frac{\mu}{v} \p_x\phi \right) + \frac{R\theta}{v^2} \p_x\phi.
\end{align*}
Then multiplying this equation by $ \frac{\p_x\phi}{v} $ and using \cref{sys}$ _1 $ again, one has that
\begin{equation}\label{per-est-3}
\begin{aligned}
\frac{d}{dt} \int_0^\pi \left[ \frac{\mu}{2v^2} (\p_x\phi)^2 - \frac{\psi}{v} \p_x\phi \right] dx + c_3 \norm{\p_x\phi}^2 \leq C\norm{\p_x\psi, \p_x\zeta}^2.
\end{aligned}
\end{equation}
Collecting \cref{per-est-2,per-est-3}, and using \cref{poin-1}, one has that
\begin{equation}\label{per-est-4}
\begin{aligned}
& \frac{d}{dt} \int_0^\pi \Big\{ M_1 \Big[ \frac{1}{2}\psi^2 + R\thetab \Phi\Big(\frac{v}{\vb}\Big) + \frac{R}{\gamma-1} \thetab \Phi\Big(\frac{\theta}{\thetab}\Big) \Big] + \frac{\mu}{2v^2} (\p_x\phi)^2 - \frac{\psi}{v} \p_x\phi \Big\} dx \\
& \qquad + c_4 \norm{\p_x\phi, \p_x\psi,\p_x\zeta}^2 + c_4 \norm{\phi, \psi,\zeta}^2 \leq 0. 
\end{aligned}
\end{equation}
If $ M_1>0 $ is large enough, the terms in $ \{ \cdots \} $ in \cref{per-est-4} satisfy that
\begin{equation}
c_5^{-1} \norm{\phi, \psi, \zeta, \p_x\phi}^2 \leq 
\int_0^\pi \{\cdots\} dx \leq c_5 \norm{\phi, \psi, \zeta, \p_x\phi}^2,
\end{equation}
which implies that
\begin{equation}
\norm{\phi, \psi, \zeta, \p_x\phi}^2(t) \leq \left(\norm{\phi_0}_1^2+ \norm{\psi_0,\zeta_0}^2\right) e^{-\frac{c_4}{c_5} t}.
\end{equation}
Since the estimates for the higher order derivatives are standard and their exponential decay rates can be proved in the same way with the aid of Poincar\'{e} inequality, thus we omit the details.

\end{proof}

\vspace{.3cm}


\bibliographystyle{amsplain}

\end{document}